\documentclass[12pt]{amsart}
\usepackage{dsfont}
\usepackage{mathrsfs} 
\usepackage{amssymb, changes}
\usepackage{amsmath}
\usepackage{amsfonts}
\usepackage{bbm}
\usepackage{amsthm}
\usepackage{amsbsy}
\usepackage{amsgen}
\usepackage{amscd}
\usepackage{amsopn}
\usepackage{amstext}
\usepackage{amsxtra}
\usepackage[english]{babel}
\usepackage[most]{tcolorbox}
\usepackage{xcolor}
\usepackage{times}
\usepackage[foot]{amsaddr}
\usepackage{hyperref}\hypersetup{
    colorlinks=true,
    linkcolor=purple,
    filecolor=blue,      
    urlcolor=cyan,
    }
   
\usepackage[margin=40pt]{geometry}

\definecolor{lime}{HTML}{A6CE39}
\DeclareRobustCommand{\orcidicon}{%
	\begin{tikzpicture}
	\draw[lime, fill=lime] (0,0) 
	circle [radius=0.16] 
	node[white] {{\fontfamily{qag}\selectfont \tiny ID}};
	\draw[white, fill=white] (-0.0625,0.095) 
	circle [radius=0.007];
	\end{tikzpicture}
	\hspace{-2mm}
}
\foreach \x in {A, ..., Z}{%
	\expandafter\xdef\csname orcid\x\endcsname{\noexpand\href{https://orcid.org/\csname orcidauthor\x\endcsname}{\noexpand\orcidicon}}
}

\newtheorem{theorem}{Theorem}[section]
\newtheorem{assumption}{Assumption}
\newtheorem{lemma}[theorem]{Lemma}
\newtheorem{corollary}[theorem]{Corollary}
\newtheorem{proposition}[theorem]{Proposition}

\tcolorboxenvironment{theorem}{
  enhanced,
  breakable,
  colback=gray!3,
  colframe=black!50!black,
  boxrule=0.8pt,
  arc=2mm,
  left=6pt,
  right=6pt,
  top=6pt,
  bottom=6pt,
  before skip=10pt,
  after skip=10pt
}

\tcolorboxenvironment{lemma}{
  enhanced,
  breakable,
  colback=gray!3,
  colframe=black!40!black,
  boxrule=0.8pt,
  arc=2mm,
  left=6pt,
  right=6pt,
  top=6pt,
  bottom=6pt,
  before skip=10pt,
  after skip=10pt
}

\tcolorboxenvironment{proposition}{
  enhanced,
  breakable,
  colback=gray!3,
  colframe=black!50!black,
  boxrule=0.8pt,
  arc=2mm,
  left=6pt,
  right=6pt,
  top=6pt,
  bottom=6pt,
  before skip=10pt,
  after skip=10pt
}

\newcommand{\E}{\mathbb{E}}
\renewcommand{\P}{\mathbb{P}}

\def\F{{\cal F}}

\newcommand{\R}{\mathbb{R}}
\newcommand{\Z}{\mathbb{Z}}
\newcommand{\N}{\mathbb{N}}
\def\eps{\varepsilon}
\def\phi{\varphi}
\newcommand{\rmd}{\,\mathrm{d}}

\def\cal{\mathcal}

\def\1{\mathds{1}}
\def\1vec{\mathbf{1}}
\def\0{\mathbf{0}}

\renewcommand{\Re}{\operatorname{Re}}
\newcommand{\convd}{\xrightarrow{d}}
\newcommand{\convp}{\xrightarrow{p}}

\DeclareMathOperator{\tr}{tr}
\DeclareMathOperator{\diag}{diag}
\newcommand{\norm}[1]{\left\| #1 \right\|}

\begin{document}
\title{Limit theorems for random walks with spatio-temporal drift}
\author[]{Ngo P. N. Ngoc\orcidB{}}
\address[N. P. N. N.]{Institute of Research and Development, and  
Faculty of Natural Sciences, Duy Tan University, Da Nang 550000, Vietnam}

\email{ngopnguyenngoc@duytan.edu.vn}
\author[]{Tuan-Minh Nguyen\orcidA{}}
\address[T.-M. N.]{School of Mathematics, Monash University, 9 Rainforest Walk, Clayton 3800, Victoria, Australia}
\email{tuanminh.nguyen@monash.edu}
\begin{abstract}
We study a class of discrete-time random walks in $\mathbb{R}^d$ whose conditional drift decays polynomially in time and grows polynomially with the distance from the origin to the current position. This class is related to several models of self-interacting random processes. We determine the asymptotic behavior of the walk under the assumption that its increments have moments of order $p$ for some $p>2$. In the linear case, where the drift depends linearly on the current position, we establish a phase transition in the convergence in distribution of the normalized process to Gaussian limits. In the nonlinear case, we identify three distinct regimes separated by a critical line and show that the normalized process exhibits qualitatively different behaviors in each regime, including convergence in distribution to a Gaussian law, convergence to a non-Gaussian limit given by the stationary distribution of a stochastic differential equation, and almost sure localization on a hypersphere.
\end{abstract}
\maketitle

\tableofcontents

\section{Introduction}
\subsection{Description of the model and main results}
Let $(X_n)_{n\ge1}$ be a sequence of random variables in $\R^d$ (not necessarily independent or identically distributed). Let $(S_n)_{n\ge 0}$ be the random walk with increments $(X_n)_{n\ge 1}$, i.e.
$$S_0=0\quad\text{and} \quad S_n:=\sum_{k=1}^n X_k \text{ for $n\ge 1$}.$$
Set $\mathcal{F}_0=\{\varnothing,\Omega\}$ and $\mathcal{F}_n=\sigma(X_1,\ldots,X_n)$ for $n\ge1$.
Let $\Z_+$ be the set of non-negative integers and $\N=\mathbb{Z}_+\setminus\{0\}$ be the set of natural numbers. Let $\|\cdot\|$ and $\langle\cdot,\cdot\rangle$ denote the Euclidean norm and scalar product; for matrices we use the induced operator norm. Assume that there exist a Borel measurable function $\mu: \R^d\times \Z_+\to \R^d$ and a $d\times d$ positive definite matrix $\Sigma$ such that for each $n\ge 0$,
$$\E[X_{n+1}|\mathcal{F}_n]=\mu(S_n,n)\quad \text{and} \quad \E[X_{n+1} X_{n+1}^\top|\mathcal{F}_n]=\Sigma+\mathcal{E}_n,$$
where $\mathcal{E}_n$ is a random $d\times d$ matrix such that it is $\mathcal{F}_n$-measurable and $\lim_{n\to\infty}\E[\|\mathcal{E}_n\|]= 0.$

We call $(S_n)_{n\ge 0}$ a \textbf{random walk with spatio-temporal drift} (RWSD), and denote it by ${\rm RWSD}(\mu,\Sigma)$. In this paper, we focus on the case where the conditional drift $\mu$ decays polynomially in time and grows polynomially with the distance from the origin to the current location of the walk, i.e.
$$\mu(s,n)=\frac{1}{n^{\beta}}\|s\|^{\alpha-1} As \quad\text{for $s\in \R^d$ and $n\ge 1$},$$
where $\alpha,\beta \ge 0$, and $A$ is a $d\times d$ matrix whose eigenvalues have positive real parts. 

Our model extends the one-dimensional random walks on $\mathbb R_+$ with drift $\mu(s,n)=\rho\, s^{\alpha}/n^{\beta}$ studied by Menshikov and Volkov~\cite{MV2008}. In \cite{MV2008}, the authors established criteria for transience and recurrence in the cases $\alpha<2\beta-1$ and $\alpha>2\beta-1$.
In the case $\alpha=\beta=1$, the corresponding nearest-neighbor model on $\mathbb Z$ is also known in the literature as the \emph{elephant random walk}, introduced by Sch\"utz and Trimper~\cite{ST2004}. The terminology is motivated by the proverbial memory of an elephant (``an elephant never forgets''): at time $n$, from its current position $s$, the elephant selects one of its previous increments uniformly at random and then either repeats it with probability $p_{s,n}=p$, or takes the opposite increment with probability $1-p$. See also \cite{CGS2017, CGS2017b, B2018, BL2019, B2022, GLRS2025, Q2025} for rigorous results and further studies of the asymptotic behavior of this model and its extension to $\Z^d$. We note that when $p_{s,n}$ depends on $s$ and $n$, the model becomes the $\text{RWSD}(\mu,\sigma^2)$ on $\Z$ with $\mu(s,n)=(2p_{s,n}-1)s/n$  and $\sigma^2=1$.  An important extension of the Sch\"utz--Trimper model to $\mathbb R$, in which the step distribution may be arbitrary, is the \emph{step-reinforced random walk}, introduced by Bertoin~\cite{B2020b,B2021, B2020a, B2024}. The case when $\alpha=-1$ and $\beta=0$, the model is known as the \textit{Lamperti's random walk}, originally investigated by Lamperti \cite{L1960, L1962, L1963} in 1960s. See \cite{MPW2017, DKW2025} for excellent monographs related to this model. A continuous-time diffusion counterpart of the Menshikov--Volkov model was later studied by Gradinaru and Offret~\cite{GO2013,O2014}. It describes a one-dimensional Brownian motion dynamics evolving in a time-dependent potential $V_{\rho,\alpha,\beta}$:
$$
\rmd X_t =  - \frac{1}{2} \partial_x V_{\rho,\alpha,\beta}(t, X_t)\, \mathrm{d} t + \mathrm{d} B_t,
\quad
X_{t_0} = x_0,
$$
where
$$
V_{\rho,\alpha,\beta}(t,x)
:=
-\frac{2\rho}{\alpha+1}\,\frac{|x|^{\alpha+1}}{t^\beta}.
$$
Gradinaru and Offret~\cite{GO2013} obtained a corresponding classification of scaling limits for one-dimensional diffusions. Here we work in discrete time and allow dependent increments in arbitrary dimension.

Throughout this paper, we also make the following assumption:
\begin{assumption}\label{assump}
Assume that $\E[X_1]=0$ and
\begin{align*} 
\sup_{n\ge 1} \E[\|X_n\|^{p}]< \infty\quad \text{for some $p>2$.}
\end{align*} 
\end{assumption}

We study limit theorems for ${\rm RWSD}(\mu,\Sigma)$ under Assumption~\ref{assump}. We consider two distinct cases where the drift $\mu(s,n)$ is either linear or nonlinear in the location variable $s$.

\begin{assumption}\label{assump.nonlinear} Fix $0<\alpha\le \beta \le 1$ and consider the drift field 
$$H(s):=\frac{1}{\alpha+1}\nabla\|s\|^{\alpha+1}=\|s\|^{\alpha-1}s\quad\text{for }s\neq 0,\quad H(0):=0.$$
Let $\Sigma$ be a deterministic $d\times d$ positive definite matrix. Assume that for each $n\ge 1$,  $$\E\big[X_{n+1}\mid \mathcal{F}_n\big]=\mu(S_n,n):=\frac{H(S_n)}{n^{\beta}}\quad \text{and}\quad \E\big[X_{n+1} X_{n+1}^\top\mid \mathcal{F}_n\big]= \Sigma.$$
\end{assumption}

\begin{theorem}\label{thm.nonlinear} Suppose that Assumption \ref{assump} and Assumption \ref{assump.nonlinear} hold.
\begin{itemize}
    \item[a.] \textbf{Above the critical line}: Assume that $\alpha < 2\beta-1$ and $0<\alpha\le 1$. 
Then $S_n/\sqrt{n}$ converges in distribution to $\mathcal{N}(0, \Sigma)$.
   \item[b.] \textbf{On the critical line}: Assume that $\beta=\frac{\alpha+1}{2}$, $0<\alpha<1$. Then ${S_n}/{\sqrt{n}}$ converges in distribution to the unique stationary distribution of the following SDE \begin{align}\label{eq:sde_critical}
    {\rm d}X_t = \left(-\frac{1}{2}X_t + H(X_t)\right){\rm d} t +\Sigma^{1/2} {\rm d} B_t.
\end{align}
In the one-dimensional case $d=1$, the stationary density is given explicitly by
$$
p(x) = C \exp\Bigl( -\Sigma^{-1}\Bigl( \frac{x^2}{2} - \frac{2}{\alpha+1}|x|^{\alpha+1} \Bigr) \Bigr) \quad \text{for } x\in\mathbb{R},
$$
where $C$ is a normalizing constant.
   \item[c.] \textbf{Under the critical line}: Assume that $0<\alpha\le \beta<1$ and $2\beta-1<\alpha$. 
Let $\gamma:=\frac{1-\beta}{1-\alpha}$ and 
$r:=\gamma^{-1/(1-\alpha)}$.
Then $n^{-\gamma}\|S_n\|\to r$ as $n\to\infty$ almost surely. In other words, $n^{-\gamma} S_n$ eventually localizes on 
the hypersphere $\mathcal C_r:=\{x\in\mathbb R^d:\|x\|=r\}$. 
\end{itemize}
\end{theorem}
We stress that the assumption $0<\alpha\le \beta\le 1$ is a \textbf{natural} condition under
which the above limit results hold \textbf{universally} under Assumption~\ref{assump}.
Obtaining limit theorems for other ranges of $\alpha$ and $\beta$ would require
additional assumptions, such as stronger moment conditions or more specific
conditions on the distributions of the increments $(X_n)$, and this is not the
aim of the present paper. 

We next consider the linear case, where $\alpha=\beta=1$. For a square matrix $B$, let $\lambda_{\max}(B)$ and $\lambda_{\min}(B)$ denote the largest and smallest real parts of the eigenvalues of $B$, respectively. For the linear case, we assume that Assumption \ref{assump} and the following conditions are fulfilled:
\begin{assumption}\label{assump2} Assume that for each $n\ge 1$,
    $$\E\big[X_{n+1}\mid \mathcal{F}_n\big]=\mu(S_n,n):=\frac{1}{n} A S_n\quad\text{and}\quad\E\big[X_{n+1} X_{n+1}^\top\mid \mathcal{F}_n\big]= \Sigma + \mathcal{E}_n,$$ 
where $A\in  \R^{d\times d}$ is a matrix such that $\lambda_{\min}(A)>0$ and $\lambda_{\max}(A)<1$, $\Sigma$ is a deterministic $d\times d$ positive definite matrix, and $\mathcal{E}_n$ is a $\mathcal{F}_n$-measurable random matrix such that $\lim_{n\to\infty}\E[\|\mathcal{E}_n\|]=0$.  
 \end{assumption}

 For $n>0$, we define $n^{B}:=\exp(\log(n)B)$. 
Let $I_d$ denote the $d$-dimensional identity matrix. For a complex matrix $Q$, let $Q^*$ denote the complex conjugate of $Q$.

\begin{theorem}\label{thm.1} Suppose that Assumption \ref{assump} and Assumption \ref{assump2} are fulfilled. 
\begin{itemize}
\item[a.] \textbf{Subcritical case:} Assume that $\lambda_{\max}(A)<1/2$. Then $S_n/\sqrt{n}$ converges in distribution to 
$\mathcal{N}(0,\Xi)$ where $\Xi$ is the unique solution to the Lyapunov matrix equation \begin{align}\label{eq.Lypunov}
    \left( \frac{1}{2}I_d-A\right) \Xi + \Xi \left( \frac{1}{2}I_d-A^\top\right) = \Sigma.
\end{align}
\item[b.] \textbf{Super-critical case:} Assume that $\lambda_{\min}(A)>\dfrac{1}{2}$. Then there exists a non-degenerate random vector $W\in\mathbb R^d$ such that
	$n^{-A} S_n$ converges a.s. and in $L^2$ to $W$.
	Furthermore,
     $\big(S_n-n^{A}W\big)/{\sqrt{n}}$ converges in distribution to $\mathcal{N}(0,-\Xi)$, where $\Xi$ solves~\eqref{eq.Lypunov}.
\item[c.] \textbf{Critical case:} Assume that $A$ has the Jordan form $A=Q(\lambda I_d + N)Q^{-1}$ where ${\rm Re}(\lambda)=1/2$ and $N$ is the nilpotent matrix with ones on the superdiagonal and  zeros elsewhere. Let \begin{align*}
    D_n & = Q\operatorname{diag}\big((\log n)^{d-\frac{1}{2}},\cdots, (\log n)^{\frac{3}{2}}, (\log n)^{\frac{1}{2}}\big)Q^{-1}.
\end{align*} Then 
$D^{-1}_n n^{-A}S_n$ converges in distribution to $\mathcal{N}(0,Q\mathcal{V}Q^*),$
where
$$\mathcal{V}_{i,j}  = \frac{(-1)^{i+j}}{(d-i)!(d-j)!(2d-i-j+1)} (Q^{-1}\Sigma (Q^{-1})^*)_{d,d}.
$$
     \end{itemize}
\end{theorem}

In parts~(a) and~(b), the solution $\Xi$ to the Lyapunov equation \eqref{eq.Lypunov} is uniquely  defined by $$\operatorname{vec}(\Xi)
=
(I_d\otimes B+B\otimes I_d)^{-1}\operatorname{vec}(\Sigma)\quad\text{with}\quad B:=\frac{1}{2}I_d-A,$$ 
where $\otimes$ is the Kronecker product and $\operatorname{vec}(A)$ denotes the vector obtained by stacking the columns of $A$. We present the proof of Theorem \ref{thm.1} in Section \ref{sec:thm1}.

When $A$ is any $d\times d$ matrix such that $0<\lambda_{\min}(A)\le \lambda_{\max}(A)<1$, we obtain the following general result.  We define
\begin{align*}
E_s := \bigoplus_{\Re(\lambda)<\frac12} E(\lambda),\quad
E_c := \bigoplus_{\Re(\lambda)=\frac12} E(\lambda),\quad
E_u := \bigoplus_{\Re(\lambda)>\frac12} E(\lambda),\quad
\end{align*}
where $E(\lambda)$ denotes the (real) generalized eigenspace of $A$ corresponding to $\lambda$ (with each complex conjugate pair grouped together and counted once). Then
\begin{equation}\label{eq:linear-spectral-decomposition}
\R^d=E_s\oplus E_c\oplus E_u.
\end{equation}
and each of these subspaces is invariant under $A$.
Let $P_s,P_c,P_u$ denote the corresponding projections. These projections need not be orthogonal. We represent restrictions in fixed orthonormal bases of the invariant subspaces, and interpret each projected covariance in the corresponding coordinates. The projections commute with $A$, and we write $A_s$, $A_c$, and $A_u$ for the restrictions of $A$ to $E_s$, $E_c$, and $E_u$, respectively.
Assume that the restriction $A_c$ of $A$ to the subspace $E_c$ has the Jordan decomposition
$$
A_c=QJ_cQ^{-1} \quad\text{with}\quad
J_c=\bigoplus_{r=1}^{R}(\lambda_r I_{m_r}+N_{m_r}),
\quad \operatorname{Re}(\lambda_r)=\frac12,
$$
where $N_{m_r}$ is the nilpotent matrix with ones on the superdiagonal and zeros elsewhere.
Define
$$
D_n = Q\bigoplus_{r=1}^{R} \operatorname{diag}\bigl( (\log n)^{m_r-\frac12},\ldots,(\log n)^{\frac32},(\log n)^{\frac12} \bigr)Q^{-1}.
$$
\begin{theorem}\label{thm.2} Suppose that Assumption \ref{assump} and Assumption \ref{assump2} are fulfilled. 
Then $n^{-A} P_u S_n$ converges to an almost‑sure limit ${W}_u$.  Furthermore, we have
$$
\begin{pmatrix}
\displaystyle \frac{P_sS_n}{\sqrt n} \\[2mm]
\displaystyle D_n^{-1}n^{-A_c}P_cS_n \\[2mm]
\displaystyle \frac{P_uS_n-n^{A_u}W_u}{\sqrt n}
\end{pmatrix}
\xrightarrow[n\to\infty]{d}
\mathcal N\left(
0,
\begin{pmatrix}
\Xi_s & 0 & 0 \\
0 & \Xi_c & 0 \\
0 & 0 & \Xi_u
\end{pmatrix}
\right),
$$
where the block‑diagonal covariance matrices are determined as follows.

\begin{itemize}
    \item $\Xi_s$ is the unique solution to the Lyapunov equation
    $$
    \bigl(\tfrac{1}{2}I - A_s\bigr) \Xi_s + \Xi_s \bigl(\tfrac{1}{2}I - A_s\bigr)^\top = P_s \Sigma P_s^\top,
    $$
    where $A_s$ is the restriction of $A$ to $E_s$.

    \item  $\Xi_c$ is given by
    $$
    \Xi_c = Q \,\mathcal{V}\, Q^*,
    $$
where $\mathcal{V}$ is the block matrix  
$\mathcal{V}=(\mathcal{V}(r,s))_{1\le r,s\le R}$, such that its $(r,s)$ block is given by
\begin{align*}
    \mathcal{V}(r,s)&=(\mathcal{V}_{i,j}(r,s))_{1\le i\le m_r, 1\le j\le m_s},\\
    \mathcal V_{i,j}(r,s)
&=
\begin{cases}
\displaystyle
\frac{(-1)^{m_r+m_s-i-j}}
{(m_r-i)!\,(m_s-j)!\,(m_r+m_s-i-j+1)}
\widetilde{\Sigma}_{m_r,m_s}(r,s),
& \lambda_r=\lambda_s,\\[3mm]
0,
& \lambda_r\ne\lambda_s.
\end{cases}
\end{align*}
with $\widetilde{\Sigma}(r,s)$ being the $(r,s)$ block of size $m_r \times m_s$ of  $\widetilde{\Sigma}=Q^{-1} P_c \Sigma P_c^\top  (Q^{-1})^*$.

    \item $\Xi_u$ is the unique solution to the Lyapunov equation
    $$
    \bigl(A_u - \tfrac{1}{2}I\bigr) \Xi_u + \Xi_u \bigl(A_u - \tfrac{1}{2}I\bigr)^\top = P_u \Sigma P_u^\top,
    $$
    where $A_u$ is the restriction of $A$ to $E_u$.
\end{itemize}
\end{theorem}
We prove the joint convergence in Theorem~\ref{thm.2} at the end of Section~\ref{sec:thm1}.

\subsection{Structure of the paper and strategies of proofs}
The remainder of the paper is organized as follows. In Section~\ref{sec:nonlinear}, we present the proof of Theorem~\ref{thm.nonlinear}. The three regimes require different arguments, according to the balance between the drift and the martingale noise. The main difficulties are the lack of regularity of the drift field at the origin and the need to work under the unconditional moment condition in  Assumption~\ref{assump}. 
\begin{itemize}
    \item Above the critical line, the drift $\mu(S_n,n)$ is asymptotically negligible
at the diffusive scale. In this regime, we show that the walk behaves essentially
as a martingale with small predictable perturbations. By constructing a suitable
Lyapunov function, we obtain tightness, and the asymptotic fluctuations are then
identified by applying a multivariate martingale central limit theorem.

    \item On the critical line, the drift $\mu(S_n,n)$ and the noise $U_{n+1}:=X_{n+1}-\E[X_{n+1}\mid\mathcal{F}_n]$ act at the same scale. Introducing the rescaled process $G_n:=n^{-1/2}S_n$, we
interpret the discrete-time dynamics of $(G_n)$ as an Euler-Maruyama type approximation scheme of the stochastic differential equation~\eqref{eq:sde_critical}. However, standard results from stochastic approximation theory for SDE do not apply in this setting. The main obstruction is that the drift field of the limiting SDE is non‑Lipschitz at the origin, which violates the regularity assumptions typically required in classical limit theorems for SDE stochastic approximation algorithms (see, e.g., \cite{D1996,P1998b}). 
To overcome this difficulty, we develop a direct diffusion‑approximation approach tailored to the present non‑Lipschitz setting: we first show that the rescaled process $(G_n)$ converges on finite logarithmic time windows to the diffusion \eqref{eq:sde_critical}, and we then exploit the long‑time ergodic properties of this diffusion to conclude that $G_n$ converges in distribution to its unique stationary measure.

    \item Below the critical line, the situation is even more delicate. In this regime,
the drift becomes strong enough to dominate the noise at large scales. Writing $G_n:=n^{-\gamma}S_n$, with $\gamma:=\frac{1-\beta}{1-\alpha}$, we interpret the discrete-time dynamics of this rescaled process as a Robbins--Monro approximation scheme for an ordinary differential equation. However, classical convergence results for ODE stochastic approximation algorithms are no longer applicable as the vector field of the ODE remains non‑Lipschitz and unbounded, and it possesses an unstable equilibrium at the origin. These features violate the
smoothness and boundedness assumptions required by standard convergence theorems, and existing martingale non‑convergence results (see, e.g., \cite{D1996, B1999}) do not suffice to rule out attraction to this unstable point. In the spirit of Pemantle’s nonattracting-point technique (see Chapter~3 in \cite{P1988}), we analyze the radial component of the dynamics and show that $\|G_n\|$ converges almost surely to the set of equilibria $\{0,r\}$, where
$r:=\gamma^{-1/(1-\alpha)}$. To exclude convergence to the unstable equilibrium at
$0$, we then extend a martingale non‑convergence theorem recently established by
Raimond and Tarrès \cite{RT2023} to the present discrete‑time setting (see Proposition \ref{lem:discrete_antitrap_from_CT}). This non-convergence result implies that $(G_n)$ cannot converge to $0$ and yields almost‑sure
localization of $n^{-\gamma}S_n$ on the hypersphere of radius $r$.
\end{itemize}
In Section~\ref{sec:thm1}, we prove Theorem~\ref{thm.1} and Theorem~\ref{thm.2} dealing with the special case $\alpha=\beta=1$. The proof relies on a martingale approach combined with a careful
matrix‑analytic study of the linear normalization induced by the drift matrix $A$. By introducing a suitable matrix‑valued normalization $\Phi(n)$, we show that the renormalized process $M_n:=\Phi(n)^{-1}S_n$ is a martingale. The asymptotic behavior is then obtained by establishing convergence of the corresponding conditional quadratic variations and verifying a Lindeberg condition. This yields central limit theorems in the subcritical and critical regimes and almost‑sure convergence in the supercritical regime, with logarithmic corrections in the critical case reflecting the Jordan structure of the matrix
$A$. For the  proof of Theorem~\ref{thm.2}, we verify the vanishing mixed covariances required for the joint limit on the invariant subspaces.

\subsection*{AI disclosure}
We used Microsoft 365 Copilot (GPT 5.6) for language editing, polishing and assistance with proof checking. The research ideas and mathematical arguments were developed entirely by the authors. We take full responsibility for the content of the paper.

\section{Proof of Theorem \ref{thm.nonlinear}}\label{sec:nonlinear}
\subsection{Above the critical line} Throughout this subsection, we assume that $\alpha<2\beta-1$ and $0<\alpha\le 1$. In particular,  we have $\beta>\frac{\alpha+1}{2}\ge \alpha.$
    
\begin{lemma}\label{Sb.lem} 
	Assume that $\alpha < 2\beta-1, 0<\alpha\le 1$. 
	Then, for each $\eps > 0$, 
	there exists a (random) positive constant $C$ such that a.s.
	$$\|S_n\| \le C n^{\frac{1}{2}+\eps}.$$ 
\end{lemma}

\begin{proof}
	Let $Z_n = \|S_n\|^2 / n^{1+2\eps}$. We shall apply the 
	Robbins--Siegmund theorem (see Proposition~\ref{RS.thm} in the appendix) to show that $(Z_n)$ is bounded almost surely.	We notice that
	\begin{align*}
		\E[\|S_{n+1}\|^2\mid\mathcal{F}_n]
		&= \|S_n\|^2 + 2\langle S_n,\,\E[X_{n+1}\mid\mathcal{F}_n]\rangle 
		+ \E[\|X_{n+1}\|^2\mid\mathcal{F}_n]\\
		&= \|S_n\|^2 + \frac{2\|S_n\|^{\alpha+1}}{n^\beta} + \tr(\Sigma),
	\end{align*}
	where we used the assumptions that $\E[X_{n+1}\mid \mathcal{F}_n]=\mu(S_n,n)=\|S_n\|^{\alpha-1}S_n/n^\beta$ and 
	$\E[X_{n+1}X_{n+1}^\top\mid\mathcal{F}_n]= \Sigma$. Dividing by $(n+1)^{1+2\eps}$, we obtain that
	\begin{equation}\label{eq:Zrec}
		\E[Z_{n+1}\mid\mathcal{F}_n]
		=\frac{n^{1+2\eps}}{(n+1)^{1+2\eps}}\,Z_n
		+ 2\frac{\|S_n\|^{\alpha+1}}{n^\beta(n+1)^{1+2\eps}}
		+ \frac{\tr(\Sigma)}{(n+1)^{1+2\eps}}.
	\end{equation}

Since $\beta\le1$ and $\alpha<2\beta-1$, we have $0<\alpha<1$. Since $\|S_n\|^2 = n^{1+2\eps}Z_n$, we have
	\begin{equation} \label{eq:Young}
	    \begin{aligned}
	      \frac{\|S_n\|^{\alpha+1}}{n^\beta(n+1)^{1+2\eps}}&\le 
	\frac{Z_n^{(\alpha+1)/2}}{n^{\beta + (1+2\eps)(1-\alpha)/2}} \le \frac{1+2\eps}{4n}\,Z_n 
		+ \frac{C_\alpha}{n^{\frac{2\beta-1-\alpha}{1-\alpha}+1+2\eps}}
	    \end{aligned}
	\end{equation}
 with some constant 
    $C_\alpha > 0$ depending only on $\alpha$,
in which the last bound is obtained by applying Young's inequality $AB \le \dfrac{A^p}{p} + \dfrac{B^q}{q}$ 
	with $p = \dfrac{2}{\alpha+1}$, $q = \dfrac{2}{1-\alpha}$, 
	$A = \mu\, Z_n^{(\alpha+1)/2}$ and $B = \mu^{-1}n^{-[\beta + (1+2\eps)(1-\alpha)/2]}$, 
and	choosing $\mu^{2/(\alpha+1)} = \dfrac{1+2\eps}{4n}$. Note that since $\alpha < 2\beta-1$, the exponent 
	$\frac{2\beta-1-\alpha}{1-\alpha}+1+2\eps > 1$, 
	so $\sum_n n^{-\frac{2\beta-1-\alpha}{1-\alpha}-1-2\eps} < \infty$. Using the Taylor expansion, we have 
	\begin{align}\label{eq.Taylor}
	    \frac{n^{1+2\eps}}{(n+1)^{1+2\eps}} 
	= 1 - \frac{1+2\eps}{n+1} + O(n^{-2}).
	\end{align}
Combining \eqref{eq:Young}, \eqref{eq.Taylor} together with \eqref{eq:Zrec}, we obtain that
	$$
	\E[Z_{n+1}\mid\mathcal{F}_n]
	\le (1 + \alpha_n)\,Z_n + \beta_n - Y_n,
	$$
	where 
	\begin{align*}
		\alpha_n = O(n^{-2}),\quad 
		\beta_n  = \frac{2C_\alpha}{n^{\frac{2\beta-1-\alpha}{1-\alpha}+1+2\eps}} 
		+ \frac{\tr(\Sigma)}{(n+1)^{1+2\eps}},\quad
		Y_n =\frac{(1+2\eps)(n-1)}{2n(n+1)}Z_n.
	\end{align*}
It is clear that all the conditions of Proposition~\ref{RS.thm} are fulfilled:
 $\sum_{n=1}^{\infty} \alpha_n < \infty$,  $\sum_{n=1}^{\infty} \beta_n < \infty$ and $Y_n \ge 0$ for all $n\ge 1$. By Proposition~\ref{RS.thm}, $Z_n$ converges a.s. to a finite limit, which proves the lemma.

\end{proof}

\begin{proof}[Proof of Theorem \ref{thm.nonlinear}.a]
Let $U_1=X_1$ and for $n\ge 1$ let 
$$
U_{n+1}:=X_{n+1}-\E[X_{n+1}\mid\mathcal F_n]
=X_{n+1}-\dfrac{H(S_n)}{n^{\beta}}.
$$
Then $M_n:=\sum_{k=1}^n U_k$
is a martingale and we obtain the Doob's decomposition:
$$
S_n = M_n + \sum_{k=1}^{n-1}\frac{H(S_{k})}{k^{\beta}}.
$$
Since $\alpha<2\beta-1$, we choose a sufficiently small $\eps>0$ such that
$
\frac{\alpha}{2}-\beta+\alpha\eps<-\frac12.
$
In particular,
$
\alpha-2\beta+2\alpha\eps<-1.
$
By Lemma \ref{Sb.lem}, we have 
$
\|H(S_n)\|=\|S_n\|^{\alpha}\le C_1 n^{\frac{\alpha}{2}+\alpha\eps}
$
for some random constant $C_1>0$. Hence
\begin{align}\label{inq.noise}
      \frac{\|H(S_n)\|}{n^{\beta}}
      \le C_1 n^{\alpha/2-\beta+\alpha\eps}\to 0
      \quad\text{a.s.}
\end{align}
It follows that
$$
\Bigl\|\sum_{k=1}^{n-1}\frac{H(S_{k})}{k^{\beta}}\Bigr\|
\le C_1 \sum_{k=1}^{n-1} k^{\frac{\alpha}{2}-\beta+\alpha\eps}
=o(\sqrt{n}).
$$
Consequently,
\begin{align}\label{S.eqn}
    \frac{S_n}{\sqrt{n}} = \frac{M_n}{\sqrt{n}} + o(1)
    \quad\text{a.s.}
\end{align}
To complete the proof, it remains to prove that $M_n/\sqrt{n}$ converges in distribution to $\mathcal{N}(0,\Sigma)$. Let
$$
\Lambda_n:=\sum_{k=1}^n\E[U_kU_k^{\top}]
=\E[M_nM_n^{\top}].
$$
We first show that
\begin{align}\label{eq.Lambda}
    \Lambda_n = n \Sigma + o(n)
    \quad\text{as } n\to\infty.
\end{align}
From the definition of $U_{k+1}$, for $k\ge1$,
$$
\E[U_{k+1}U_{k+1}^{\top}\mid\mathcal{F}_k]
= \E[X_{k+1}X_{k+1}^{\top}\mid\mathcal{F}_k]
   -\frac{1}{k^{2\beta}}H(S_k)H(S_k)^{\top}
   =\Sigma + \delta_k,
$$
where
$
\delta_k:= -{k^{-2\beta}}H(S_k)H(S_k)^{\top}.
$
Taking expectations, we get
$$
\E[U_{k+1}U_{k+1}^{\top}]=\Sigma+\E[\delta_k].
$$
Since
$$
\|\delta_k\|=k^{-2\beta}\|H(S_k)\|^2=k^{-2\beta}\|S_k\|^{2\alpha},
$$
we have
$
\E\|\delta_k\|
=
{k^{-2\beta}}\E\|S_k\|^{2\alpha}
\le
{k^{-2\beta}} \bigl(\E\|S_k\|^p\bigr)^{2\alpha/p}.
$
We also have 
\begin{align}\label{inq.Lp}
\E\|S_k\|^p
\le k^{p-1}\sum_{j=1}^k \E\|X_j\|^p
\le C_2 k^p,
\end{align}
where  $C_2:=\sup_{n\ge 1} \E[\|X_n\|^p]<\infty$ by Assumption \ref{assump}. 
Hence
$
\E\|\delta_k\|
\le C_2^{2\alpha/p} k^{2\alpha-2\beta}\to0.
$
Here we used that $\alpha<\beta$, which follows from $\alpha<2\beta-1$ and $0<\alpha\le1$. Therefore, by Ces\`aro's lemma,
$$
\frac1n\sum_{k=1}^{n-1}\E[\delta_k]\to0.
$$
Since the contribution of $U_1$ is fixed, this proves \eqref{eq.Lambda}. 

It follows from \eqref{eq.Lambda} that $\Lambda_n$ is positive definite for all sufficiently large $n$. To prove the convergence to $\mathcal{N}(0,\Sigma)$ of the martingale $(M_n)_{n\ge 1}$, we apply Corollary~\ref{thm.mtg2} (see in the appendix). We now verify the conditions of this corollary. 

From Lemma \ref{Sb.lem}, we have almost surely
$
\|\delta_k\|
\le C_1^2 k^{\alpha-2\beta+2\alpha\eps}.
$
Recall that $\eps$ is sufficiently small such that
$
\alpha-2\beta+2\alpha\eps<-1.
$
Thus
$
\sum_{k=1}^\infty\|\delta_k\|<\infty$ a.s.
Hence, a.s.,
$$
\sum_{k=1}^n\E[U_kU_k^{\top}\mid\mathcal{F}_{k-1}]
= n\Sigma + \sum_{k=1}^{n-1}\delta_{k}+O(1)
=n\Sigma+o(n),
$$
where the term $O(1)$ comes from the initial variable $U_1$. Therefore
$$
\Lambda_n^{-1/2}\sum_{k=1}^n\E[U_kU_k^{\top}\mid\mathcal{F}_{k-1}]\Lambda_n^{-1/2}
=
\Lambda_n^{-1/2}\bigl(n\Sigma+o(n)\bigr)\Lambda_n^{-1/2}
\xrightarrow{\P} I_d,
$$
since $\Lambda_n= n\Sigma+o(n)$ by \eqref{eq.Lambda}. 

We next show that for every $\eps>0$,
$$
\Lambda_n^{-1/2}\sum_{k=1}^n\E\bigl[U_kU_k^{\top}
\mathbf{1}_{\{\|\Lambda_n^{-1/2}U_k\|>\eps\}}\mid\mathcal{F}_{k-1}\bigr]\Lambda_n^{-1/2}
\xrightarrow{\P} 0.
$$
By conditional Jensen's inequality,
$$
\E\|U_k\|^p
\le2^{p-1}\left(\E\|X_k\|^p+\E\|\E[X_k\mid\mathcal F_{k-1}]\|^p\right)
\le2^p\E\|X_k\|^p,
$$
with the same bound for $U_1=X_1$. Hence $\sup_k\E\|U_k\|^p<\infty$.
Therefore $\|U_k\|^2$ is uniformly integrable.  Note that $\|\Lambda_n^{-1/2}\|=O(n^{-1/2})$ by \eqref{eq.Lambda}. Hence
$
{\eps}/{\|\Lambda_n^{-1/2}\|}\to\infty.
$
Using $\|U_kU_k^\top\|=\|U_k\|^2$, we obtain
\begin{align*}
&\E\Bigl\|
\Lambda_n^{-1/2}\sum_{k=1}^n\E\bigl[
U_kU_k^{\top}
\mathbf{1}_{\{\|\Lambda_n^{-1/2}U_k\|>\eps\}}
\mid\mathcal{F}_{k-1}\bigr]\Lambda_n^{-1/2}
\Bigr\|\\
&\le
\|\Lambda_n^{-1/2}\|^2
\sum_{k=1}^n
\E\bigl[
\|U_k\|^2
\mathbf{1}_{\{\|U_k\|>\eps/\|\Lambda_n^{-1/2}\|\}}
\bigr]\le
C n^{-1}
\sum_{k=1}^n
\E\bigl[
\|U_k\|^2
\mathbf{1}_{\{\|U_k\|>\eps/\|\Lambda_n^{-1/2}\|\}}
\bigr]\to0,
\end{align*}
where the last convergence follows from the uniform integrability of
$\{\|U_k\|^2:k\ge1\}$. This implies the required convergence in probability.

As both conditions of Corollary~\ref{thm.mtg2} are satisfied, we infer that
$
\Lambda_n^{-1/2}M_n \xrightarrow{d} \mathcal{N}_d(0,I_d).
$
Since $n^{-1/2}\Lambda_n^{1/2}\to \Sigma^{1/2}$, we have
$
n^{-1/2}M_n \xrightarrow{d}\mathcal{N}(0,\Sigma).
$
Combining this with \eqref{S.eqn}, we obtain
$
n^{-1/2}{S_n}\xrightarrow{d}\mathcal{N}(0,\Sigma),
$
which completes the proof.
\end{proof}

\subsection{On the critical line}

Recall that $\beta=(\alpha+1)/2$. Define $G_n = S_n/\sqrt{n}$. Let $U_1=X_1$ and for $n\ge 1$ let $$U_{n+1}:=X_{n+1}-\E[X_{n+1}\mid\mathcal{F}_n]=X_{n+1}-\frac{ H(S_n)}{n^{(\alpha+1)/2}}=X_{n+1}-\frac{H(G_n)}{\sqrt n}.$$ From the recursion $S_{n+1}=S_n+X_{n+1}$,
we obtain
\begin{align*}
G_{n+1} &= \frac{\sqrt{n}}{\sqrt{n+1}}G_n + \frac{X_{n+1}}{\sqrt{n+1}} =   \frac{\sqrt{n}}{\sqrt{n+1}}G_n + \frac{H(G_n)}{\sqrt{n(n+1)}} + \frac{U_{n+1}}{\sqrt{n+1}} \\
&= \Bigl(1-\frac{1}{2(n+1)}-\frac{1}{8(n+1)^2}+O(n^{-3})\Bigr)G_n\\
& + \Bigl(\frac{1}{n+1} + \frac{1}{2(n+1)^2} + O(n^{-3})\Bigr) \|G_n\|^{\alpha-1} G_n + \frac{U_{n+1}}{\sqrt{n+1}},
\end{align*}
where in the last step we use the Taylor expansions for $\sqrt{n/(n+1)}$ and $1/\sqrt{n(n+1)}$ as $n\to\infty$. 
Therefore
\begin{equation}\label{eq:Gn_recursion}
G_{n+1}-G_n = \frac{1}{n+1}\Bigl(-\frac12 G_n + H(G_n) + R_n\Bigr) + \frac{U_{n+1}}{\sqrt{n+1}},
\end{equation}
where $$R_n := -\frac{1}{8(n+1)}G_n + \frac{1}{2(n+1)}\|G_n\|^{\alpha-1}G_n + O\!\left(n^{-2}\right)(\|G_n\| + \|G_n\|^{\alpha}).$$ Note that
$\|R_n\| \le C n^{-1}(\|G_n\|+\|G_n\|^{\alpha})$ for some constant $C>0$.

We note that \eqref{eq:Gn_recursion} defines a stochastic approximation for the diffusion \eqref{eq:sde_critical}, whose drift vector field
\[
F(x)=-\tfrac{1}{2}x+\|x\|^{\alpha-1}x\quad\text{with } {0<\alpha<1}.
\]
is not Lipschitz on any neighborhood of the origin.
This prevents us from applying directly the standard results on the convergence of stochastic approximation algorithms to stationary distributions of SDEs, which typically rely on assumptions such as global Lipschitz continuity of the drift and a conditional $L^p$-moment bound on the martingale noise for some $p>2$ (see, e.g., \cite{D1996,P1998b}). Instead, we use a direct diffusion-approximation argument adapted to the present non-Lipschitz setting. More precisely, we show that the rescaled process $(G_n)$ converges on finite logarithmic time windows to the diffusion generated by $F$, and then exploit the long-time ergodic behavior of this diffusion to conclude that $G_n$ converges in distribution to its unique stationary measure.

\begin{lemma}\label{sde.unique_solution} For every initial law on $\R^d$, the SDE \eqref{eq:sde_critical} admits a
non-explosive global weak solution, and its law is unique. Let
$(P_t)_{t\ge0}$ denote the corresponding Markov semigroup. Then there exists
a unique stationary distribution $\pi$ for this semigroup and there exist
constants $C,\lambda>0$ such that
\begin{equation}
\label{eq.exp_ergo}
|P_t f(x)-\pi(f)|
\le
C e^{-\lambda t}(1+\|x\|^2)\|f\|_\infty
\end{equation}
for every bounded measurable function $f$, every $x\in\R^d$ and every $t\ge0$.
\end{lemma}

\begin{proof}
	We first show that the conditions $(H_\sigma^{\alpha})$ and $(H_{b}^0)$ of Theorem~\ref{thm:MPZ} (see in the appendix) hold for the diffusion \eqref{eq:sde_critical} with
	$$
	\sigma(x):=\Sigma^{1/2},\quad 
	b(x):=-\frac{1}{2}x+\|x\|^{\alpha-1}x.
	$$
	It is clear that the condition $(H^\alpha_\sigma)$ is fulfilled as $\sigma$ is a constant and
	symmetric positive definite matrix. Since
	$H(x):=\|x\|^{\alpha-1}x$ is H\"older continuous with exponent
	$\alpha\in(0,1)$, there exists a constant $L_\alpha>0$ such that
	$
	\|H(x)-H(y)\|\le L_\alpha\|x-y\|^\alpha$ for all $x,y\in\mathbb R^d$, and thus
	$$
	\|b(x)-b(y)\|
	\le \frac12\|x-y\|+L_\alpha\|x-y\|^\alpha
	\le \kappa_1(1\vee\|x-y\|)
	$$
	for some constant $\kappa_1>0$. Hence condition $(H^0_b)$ of
	Theorem~\ref{thm:MPZ} is also verified. By Theorem~\ref{thm:MPZ}, the
	SDE \eqref{eq:sde_critical} admits a unique weak solution and a
	transition density.

    We now prove non-explosion and exponential ergodicity. Applying the weighted Young inequality
	$ab \le \frac{\eps^p a^p}{p}+\frac{b^q}{\eps^q q}$
	with $p=\frac{2}{\alpha+1}$, $q=\frac{2}{1-\alpha}$,
	$a=\|x\|^{\alpha+1}$, $b=1$, and $\eps^p=\frac{1}{2}$, we notice that
\begin{equation}\label{eq:Young_used}
		\|x\|^{\alpha+1}
		\le \frac{\alpha+1}{4}\,\|x\|^2 + C_\alpha
		\quad\forall\, x \in \mathbb{R}^d, \text{ with $C_\alpha := \frac{1-\alpha}{2} \cdot 2^{\frac{1+\alpha}{1-\alpha}}$}.
	\end{equation}
Let $V(x)=\|x\|^2$ and denote by $\mathcal{L}$ the generator of~\eqref{eq:sde_critical}. Using \eqref{eq:Young_used}, we notice that  
$$
(\mathcal{L}V)(x) =-\|x\|^2+2\|x\|^{\alpha+1}+\mathrm{tr}(\Sigma) \le -\frac{1-\alpha}{2}\|x\|^2+d_0,
$$
where $d_0:=2C_\alpha+\mathrm{tr}(\Sigma)$. By Theorem~\ref{thm.Ber_nonexp} (see in the appendix), the SDE \eqref{eq:sde_critical} admits global in time solutions. Together with the uniqueness of the weak solution obtained from Theorem~\ref{thm:MPZ}, this yields a unique-in-law non-explosive global weak solution to \eqref{eq:sde_critical}. By Proposition \ref{thm.ergodicity} (see in the appendix),
the corresponding Markov semigroup admits a unique invariant probability measure $\pi$ and \eqref{eq.exp_ergo} holds.
\end{proof}

\begin{lemma}\label{lem:tightness}
We have $\sup_{n\ge 1}\E[\|G_n\|^2]<\infty$.
\end{lemma}
\begin{proof}
From (\ref{eq:Gn_recursion}), we have
\begin{align*}
\E[\|G_{n+1}\|^2\mid\mathcal{F}_n]
&=\E\Big[ \Bigl\|G_n + \frac{1}{n+1}\Bigl(-\frac12 G_n + H(G_n) + R_n\Bigr)+ \frac{1}{\sqrt{n+1}}U_{n+1}\Bigr\|^2 \mid \mathcal{F}_n \Big].
\end{align*}
Expanding the square and using the fact that $\E[U_{n+1}\mid\mathcal{F}_n]=0$, we obtain
\begin{align*}
    \E[\|G_{n+1}\|^2\mid\mathcal{F}_n]
&=\Bigl\|G_n + \frac{1}{n+1}\Bigl(-\frac12 G_n + H(G_n) + R_n\Bigr)\Bigr\|^2
+ \frac{1}{n+1}\E\Big[\|U_{n+1}\|^2 \mid \mathcal{F}_n \Big].
\end{align*}
Moreover,
$
\E[\|U_{n+1}\|^2\mid\mathcal F_n]
=
\tr(\Sigma)-n^{-1}\|H(G_n)\|^2
\le \tr(\Sigma).
$
Hence
\begin{align*}
\E[\|G_{n+1}\|^2\mid\mathcal{F}_n]
&\le \|G_{n}\|^2
+ \frac{2}{n+1}\langle G_n, -\tfrac12 G_n + H(G_n) +R_n\rangle
+ \frac{\tr(\Sigma)}{n+1} \\
&\quad
+ \frac{1}{(n+1)^2}\Bigl\|-\frac12 G_n + H(G_n) + R_n \Bigr\|^2.
\end{align*}
Note that there exists positive constants $c_1, c_2>0$ such that
\begin{align*}
    &\langle G_n, R_n \rangle
= O(n^{-1})(\|G_n\|^{2}+\|G_n\|^{\alpha+1})
\le \frac{c_1}{n}(1+\|G_{n}\|^2),\\
  & \Bigl\|-\frac12 G_n + H(G_n) + R_n \Bigr\|^2
\le c_2(1+\|G_{n}\|^2).
\end{align*}
Also, there exists a constant $c_3>0$ such that
$$
\langle x, -\tfrac12 x + H(x)\rangle
=
-\tfrac12\|x\|^2+\|x\|^{\alpha+1}
\le
-\tfrac14 V(x)+c_3
\quad\text{for all }x\in\mathbb R^d.
$$
Therefore,
\begin{align*}
\E[\|G_{n+1}\|^2\mid\mathcal{F}_n]
&\le
\|G_{n}\|^2
-\frac{1}{2(n+1)}\|G_{n}\|^2
+\frac{C_1}{n(n+1)}(1+\|G_{n}\|^2)\\
&\quad
+\frac{C_2}{n+1}
+\frac{C_3}{(n+1)^2}(1+\|G_{n}\|^2).
\end{align*}
Consequently, for all sufficiently large $n$,
$$
\E[\|G_{n+1}\|^2\mid\mathcal{F}_n]
\le
\left(1-\frac{c}{n}\right)\|G_{n}\|^2+\frac{C}{n}
$$
for some constants $c,C>0$. Taking expectations on both sides, we obtain
$$
\E[\|G_{n+1}\|^2]
\le
\left(1-\frac{c}{n}\right)\E[\|G_{n}\|^2]+\frac{C}{n}
$$
for all $n\ge n_0$, where $n_0$ is sufficiently large. By induction, one can show that $$\E[\|G_{n}\|^2]\le \max\left\{\frac{C}{c}, \E[\|G_{n_0}\|^2]\right\}$$ for all $n\ge n_0$.
This yields the claim of the lemma.
\end{proof}

For a fixed $T$, we denote by $D([0,T], \mathbb{R}^d)$ the Skorokhod space of all càdlàg functions from $[0,T]$ to $\R^d$.

\begin{lemma}\label{Aldous.tightness}
    Let $t_n:=\sum_{k=1}^n \frac1k$ and $m(n,T):=\max\{k\ge n:\ t_k-t_n\le T\}$. Fix $T>0$ and let
$$\bar G^{(n)}(t):=G_{m(n,t)}\quad\text{for }0\le t\le T.$$
Then the family $(\bar G^{(n)}(\cdot))_{n\ge1}$ is tight in $D([0,T],\R^d)$.
\end{lemma}

\begin{proof}
To prove tightness in $D([0,T], \mathbb{R}^d)$, we verify Aldous' criterion (see \cite{A1978} or Theorem 16.10 in \cite{BP1999}, p. 178):
\begin{itemize}
\item[(i)] For every $\eps>0$ there exists $R<\infty$ such that
$$
\sup_{n\ge1}\P\Bigl(\sup_{0\le t\le T}\|\bar G^{(n)}(t)\|>R\Bigr)<\eps;
$$
\item[(ii)] Let $\mathcal G_t^{(n)}$ be the filtration generated by $\bar G^{(n)}(s)$ for $0\le s\le t$. For every $\eps,\eta>0$, there exists $\delta>0$ such that for every
sequence of $(\mathcal{G}_t^{(n)})$-stopping times $\tau_n\le T$ and every sequence $\theta_n$ with
$0\le \theta_n\le \delta$ and $\tau_n+\theta_n\le T$,
$$
\limsup_{n\to\infty}
\P\bigl(\|\bar G^{(n)}(\tau_n+\theta_n)-\bar G^{(n)}(\tau_n)\|>\eps\bigr)
\le \eta.
$$
\end{itemize}

Notice that
$$
\sup_{0\le t\le T}\|\bar G^{(n)}(t)\|
=
\max_{n\le \ell\le m(n,T)}\|G_\ell\|.
$$
From the recursion \eqref{eq:Gn_recursion},
for every $n\le \ell\le m(n,T)$, we have
\begin{align*}
G_\ell
=
G_n
+\sum_{j=n}^{\ell-1}\frac{1}{j+1}\Bigl(-\frac12 G_j+H(G_j)+R_j\Bigr)
+\sum_{j=n}^{\ell-1}\frac{U_{j+1}}{\sqrt{j+1}}.
\end{align*}
Hence
\begin{align*}
\max_{n\le \ell\le m(n,T)}\|G_\ell\|
&\le
\|G_n\|
+\sum_{j=n}^{m(n,T)-1}\frac{1}{j+1}
\Bigl\|-\frac12 G_j+H(G_j)+R_j\Bigr\| 
+\max_{n\le \ell\le m(n,T)}
\left\|
M_\ell^{(n)}
\right\|,
\end{align*}
where
$$
M_\ell^{(n)}
:=
\sum_{j=n}^{\ell-1}\frac{U_{j+1}}{\sqrt{j+1}},\quad n\le \ell\le m(n,T).
$$
Then $(M_\ell^{(n)})_{n\le \ell\le m(n,T)}$ is a martingale with respect to $(\mathcal F_\ell)_{\ell\ge n}$. By Doob's
$L^2$ inequality,
\begin{align*}
\E\left[
\max_{n\le \ell\le m(n,T)}\|M_\ell^{(n)}\|^2
\right]
&\le
4\E\left[
\sum_{j=n}^{m(n,T)-1}\frac{\|U_{j+1}\|^2}{j+1}
\right].
\end{align*}
Notice that
\begin{align}\label{eq.U.cond}
    \E[\|U_{j+1}\|^2\mid\mathcal F_j]
=
\operatorname{tr}\!\left(
\E[U_{j+1}U_{j+1}^\top\mid\mathcal F_j]
\right)
=
\operatorname{tr}(\Sigma-j^{-1}H(G_j)H(G_j)^\top)
\le \operatorname{tr}(\Sigma).
\end{align}
Hence
\begin{align*}
\E\left[
\max_{n\le \ell\le m(n,T)}\|M_\ell^{(n)}\|^2
\right]
&\le
4\,\operatorname{tr}(\Sigma)
\sum_{j=n}^{m(n,T)-1}\frac1{j+1}
\le 4\,\operatorname{tr}(\Sigma)\,T.
\end{align*}
Therefore
$$
\sup_{n\ge1}
\E\left[
\max_{n\le \ell\le m(n,T)}\|M_\ell^{(n)}\|
\right]
<\infty.
$$
We next bound the drift term. Using the fact that
$\|H(x)\|=\|x\|^\alpha\le 1+\|x\|$,
and
$\|R_j\|\le C j^{-1}(\|G_j\|+\|G_j\|^\alpha)$, together with
Lemma~\ref{lem:tightness}, we have
\begin{align}\label{drift.bound}
    \sup_{j\ge1}\E\left[\left\|-\frac12 G_j+H(G_j)+R_j\right\|\right]\le C_2.
\end{align}
for some constant $C_2>0$.
Therefore
\begin{align*}
\sup_{n\ge1}
\E\left[
\sum_{j=n}^{m(n,T)-1}\frac{1}{j+1}
\Bigl\|-\frac12 G_j+H(G_j)+R_j\Bigr\|
\right]
&\le
C_2 \sup_{n\ge1}\sum_{j=n}^{m(n,T)-1}\frac1{j+1}
\le C_2 T.
\end{align*} 

Combining the previous bounds with Lemma~\ref{lem:tightness}, we get
$$
\sup_{n\ge1}
\E\left[
\sup_{0\le t\le T}\|\bar G^{(n)}(t)\|
\right]
<\infty.
$$
Using Markov's inequality, we obtain
$$
\sup_{n\ge1}\P\Bigl(\sup_{0\le t\le T}\|\bar G^{(n)}(t)\|>R\Bigr)
\le
\frac{1}{R}
\sup_{n\ge1}
\E\left[
\sup_{0\le t\le T}\|\bar G^{(n)}(t)\|
\right]
\to0
\quad\text{as } R\to\infty.
$$
Thus condition (i) holds.

Let $(\tau_n)$ be a sequence of $(\mathcal G_t^{(n)})$-stopping times with
$\tau_n\le T$, and let
$(\theta_n)$ satisfy $0\le\theta_n\le\delta$ and $\tau_n+\theta_n\le T$ for some $\delta>0$.
Define $k_n:=m(n,\tau_n)$ and 
$\ell_n:=m(n,\tau_n+\theta_n).$
Then
$$
\bar G^{(n)}(\tau_n+\theta_n)-\bar G^{(n)}(\tau_n)
=
G_{\ell_n}-G_{k_n}.
$$
The random indices $k_n$ and $\ell_n$ are discrete stopping times with respect to $(\mathcal F_j)_{j\ge n}$. Hence
$\{k_n\le j<\ell_n\}\in\mathcal F_j$ for every $j\ge n$.
From the recursion \eqref{eq:Gn_recursion}, we have
\begin{align} 
  G_{\ell_n}-G_{k_n}
&=
\sum_{j=k_n}^{\ell_n-1}\frac{1}{j+1}\Bigl(-\frac12 G_j+H(G_j)+R_j\Bigr)
+
\sum_{j=k_n}^{\ell_n-1}\frac{U_{j+1}}{\sqrt{j+1}}.
\label{eq:stopping_decomp}
\end{align}
By definition of $m(n,\cdot)$,
\begin{align}
\sum_{j=k_n}^{\ell_n-1}\frac1{j+1}
=
t_{\ell_n}-t_{k_n}
\le
(\tau_n+\theta_n)-\tau_n+\frac{1}{k_n+1}
\le
\delta+\frac1n.
\label{eq:harmonic_window}
\end{align}
Fix $\eps,\eta>0$. By condition (i), choose $R<\infty$ such that
$$
\sup_{n\ge1}\P\left(A_n^c\right)<\frac{\eta}{4} \quad\text{with}\quad
A_n:=\left\{\sup_{0\le t\le T}\|\bar G^{(n)}(t)\|\le R\right\}.
$$
On $A_n$, we have $\|G_j\|\le R$ for every $k_n\le j\le \ell_n-1$, and therefore
there exists a constant $C_R>0$ such that
$$
\Bigl\|-\frac12 G_j+H(G_j)+R_j\Bigr\|
\le C_R
\quad\text{for all }k_n\le j\le \ell_n-1.
$$
Hence, on $A_n$, by \eqref{eq:harmonic_window}, we have
$$
\left\|
\sum_{j=k_n}^{\ell_n-1}\frac{1}{j+1}\Bigl(-\frac12 G_j+H(G_j)+R_j\Bigr)
\right\|
\le
C_R\left(\delta+\frac1n\right).
$$
Choose $\delta>0$ sufficiently small such that $C_R\delta<\eps/2$. Then for all sufficiently large $n$,
\begin{align}
\P\left(
\left\|
\sum_{j=k_n}^{\ell_n-1}\frac{1}{j+1}\Bigl(-\frac12 G_j+H(G_j)+R_j\Bigr)
\right\|>\frac{\eps}{2}
\right)
\le \P(A_n^c)
<\frac{\eta}{4}.
\label{eq:drift_window_prob}
\end{align}
Using the fact that $(U_j)$ is a martingale difference sequence and
$\{k_n\le j<\ell_n\}\in \mathcal{F}_j$ together with \eqref{eq.U.cond} and \eqref{eq:harmonic_window}, we have
\begin{align*}
\E\left\|
\sum_{j=k_n}^{\ell_n-1}\frac{U_{j+1}}{\sqrt{j+1}}
\right\|^2
&= \E\left\| \sum_{j=n}^{m(n,T)-1}\mathbf 1_{\{k_n\le j<\ell_n\}}\frac{U_{j+1}}{\sqrt{j+1}}\right\|^2
\\
& =
\E\left[
\sum_{j=k_n}^{\ell_n-1}\frac{1}{j+1}
\E\bigl[\|U_{j+1}\|^2\mid\mathcal F_j\bigr]
\right] 
\le
\tr(\Sigma)\left(\delta+\frac1n\right). 
\end{align*}
Therefore, by Chebyshev's inequality, we get
\begin{align}
    \limsup_{n\to\infty}
\P\left(
\left\|
\sum_{j=k_n}^{\ell_n-1}\frac{U_{j+1}}{\sqrt{j+1}}
\right\|>\frac{\eps}{2}
\right)
\le
\frac{4\,\tr(\Sigma)}{\eps^2}\,\delta.
\label{eq:mart_window_prob}
\end{align}
Combining \eqref{eq:stopping_decomp}, \eqref{eq:drift_window_prob}, and
\eqref{eq:mart_window_prob}, we obtain
$$
\limsup_{n\to\infty}
\P\bigl(\|\bar G^{(n)}(\tau_n+\theta_n)-\bar G^{(n)}(\tau_n)\|>\eps\bigr)
\le
\frac{\eta}{4}+\frac{4\tr(\Sigma)}{\eps^2}\delta.
$$
Choosing $\delta>0$ sufficiently small, the right-hand side is bounded by $\eta$. Thus, condition (ii) holds. Therefore $(\bar G^{(n)})_{n\ge1}$ is tight in
$D([0,T],\R^d)$.
\end{proof}

\begin{lemma}\label{lem:critical_line_small_jumps}
For every $T>0$,
$$
\sup_{0\le t\le T}\|\Delta \bar G^{(n)}(t)\|
\xrightarrow[n\to\infty]{\P}0,
$$
where $\Delta \bar G^{(n)}(t):=\bar G^{(n)}(t)-\bar G^{(n)}(t-)$.
\end{lemma}

\begin{proof}
Note that the nonzero jumps of $\bar G^{(n)}$ on $[0,T]$ are among the increments
$G_{j+1}-G_j$ for $n\le j\le m(n,T)-1$.
Fix $\eps,\eta>0$. By condition~(i) in the proof of Lemma~\ref{Aldous.tightness},
there exists $R<\infty$ such that
\begin{equation}\label{eq:compact_event_smalljumps}
\sup_{n\ge1}\P\!\left(A_n^c\right)<\frac{\eta}{2}\quad\text{with}\quad A_n:=\Big\{\sup_{0\le t\le T}\|\bar G^{(n)}(t)\|\le R\Big\}.
\end{equation}
On $A_n$, we have $\|G_j\|\le R$ for all $n\le j\le m(n,T)$. Recall from \eqref{eq:Gn_recursion} that
$$
G_{j+1}-G_j
=
\frac{1}{j+1}\Bigl(-\frac12 G_j+H(G_j)+R_j\Bigr)
+\frac{U_{j+1}}{\sqrt{j+1}},
$$
where
$$
U_{j+1}=X_{j+1}-\frac{H(S_j)}{j^{(\alpha+1)/2}}
      =X_{j+1}-\frac{H(G_j)}{\sqrt j}.
$$
Since $\|H(x)\|=\|x\|^\alpha$ and $0<\alpha<1$, on $A_n$ we have
$\|H(G_j)\|\le R^\alpha$ for all $n\le j\le m(n,T)-1$. Moreover,
$\|R_j\|\le C j^{-1}(\|G_j\|+\|G_j\|^\alpha)\le C_R j^{-1}$ on $A_n$.
Therefore there exists a constant $C_R>0$ such that, on $A_n$,
$$
\|G_{j+1}-G_j\|
\le
\frac{\|X_{j+1}\|}{\sqrt{j+1}}+\frac{C_R}{j+1}
\quad\text{for all }n\le j\le m(n,T)-1.
$$
Choose $N_1$ sufficiently large such that
$
\frac{C_R}{n+1}\le \frac{\eps}{2}
$
for all $n\ge N_1$. Then, for $n\ge N_1$, on the event $A_n$,
\begin{align*}
\sup_{0\le t\le T}\|\Delta \bar G^{(n)}(t)\|
\le
\max_{n\le j\le m(n,T)-1}\|G_{j+1}-G_j\|
\le
\max_{n\le j\le m(n,T)-1}\frac{\|X_{j+1}\|}{\sqrt{j+1}}+\frac{\eps}{2}.
\end{align*}
Hence, for all $n\ge N_1$,
\begin{align}
\P\!\left(\sup_{0\le t\le T}\|\Delta \bar G^{(n)}(t)\|>\eps\right)
&\le
\P(A_n^c)
+\P\!\left(
\max_{n\le j\le m(n,T)-1}\frac{\|X_{j+1}\|}{\sqrt{j+1}}>\frac{\eps}{2}
\right).
\label{eq:split_smalljumps}
\end{align}
It remains to control the second term. By the union bound and Markov's inequality,
\begin{align*}
\P\!\left(
\max_{n\le j\le m(n,T)-1}\frac{\|X_{j+1}\|}{\sqrt{j+1}}>\frac{\eps}{2}
\right)
&\le
\sum_{j=n}^{m(n,T)-1}
\P\!\left(\|X_{j+1}\|>\frac{\eps}{2}\sqrt{j+1}\right) \\
&\le
\sum_{j=n}^{m(n,T)-1}
\frac{2^p}{\eps^p (j+1)^{p/2}}\E\|X_{j+1}\|^p.
\end{align*}
Since $\sup_{k\ge1}\E\|X_k\|^p<\infty$ by Assumption~\ref{assump},
there exists a constant $C_\eps>0$ such that
\begin{align}
\P\!\left(
\max_{n\le j\le m(n,T)-1}\frac{\|X_{j+1}\|}{\sqrt{j+1}}>\frac{\eps}{2}
\right)
\le
C_\eps \sum_{j=n}^{\infty}(j+1)^{-p/2}.
\label{eq:X_tail_smalljumps}
\end{align}
As $p>2$, the series on the right-hand side tends to $0$ as
$n\to\infty$.
Combining \eqref{eq:split_smalljumps} together with \eqref{eq:compact_event_smalljumps}
and \eqref{eq:X_tail_smalljumps}, we obtain
$$
\limsup_{n\to\infty}
\P\!\left(\sup_{0\le t\le T}\|\Delta \bar G^{(n)}(t)\|>\eps\right)
\le
\frac{\eta}{2}.
$$
Since $\eta>0$ is arbitrary, it follows that
$
\sup_{0\le t\le T}\|\Delta \bar G^{(n)}(t)\|
\xrightarrow[n\to\infty]{\P}0.
$
This completes the proof.
\end{proof}

We now identify the weak limits of $(\bar G^{(n)})$ by the martingale problem
method. Define the second-order differential operator
\begin{align}\label{L.df}
    \mathcal{L}\varphi(x)
:=
\nabla\varphi(x)\cdot\Bigl(-\frac12 x+H(x)\Bigr)
+
\frac12\tr\!\bigl(\Sigma \nabla^2\varphi(x)\bigr)
\text{ for each }
\varphi\in C_c^2(\R^d).
\end{align}

\begin{lemma}\label{lem:critical_line_generator}
Let $\varphi\in C_c^2(\R^d)$ and $T>0$. Then
$$
\sup_{0\le t\le T}
\left|
\sum_{j=n}^{m(n,t)-1}
\E\bigl[\varphi(G_{j+1})-\varphi(G_j)\mid\mathcal F_j\bigr]
-
\int_0^t \mathcal{L}\varphi(\bar G^{(n)}(s))\rmd s
\right|
\xrightarrow[n\to\infty]{L^1}0.
$$
\end{lemma}

\begin{proof}
    Fix $\varphi\in C_c^2(\R^d)$. For each $j\ge1$, let $\Delta_{j+1}:=G_{j+1}-G_j.$
Using Taylor's expansion, we have
\begin{align}
\varphi(G_{j+1})-\varphi(G_j)
=
\nabla\varphi(G_j)\cdot \Delta_{j+1}
+\frac12 \Delta_{j+1}^\top \nabla^2\varphi(G_j)\Delta_{j+1}
+r_{j+1},
\label{eq:taylor_phi}
\end{align}
where
$$
r_{j+1}
=
\int_0^1(1-\theta)\,
\Delta_{j+1}^\top
\Bigl(\nabla^2\varphi(G_j+\theta\Delta_{j+1})-\nabla^2\varphi(G_j)\Bigr)
\Delta_{j+1}\rmd \theta.
$$
Let
$$
\omega_\varphi(\eta) :=\sup_{\|x-y\|\le \eta}\|\nabla^2\varphi(x)-\nabla^2\varphi(y)\|\quad\text{for } \eta>0.
$$
Then $\omega_\varphi(\eta)\to0$ as $\eta\downarrow0$, and
\begin{align}
|r_{j+1}|
\le
\frac12\,\omega_\varphi(\|\Delta_{j+1}\|)\,\|\Delta_{j+1}\|^2.
\label{eq:r_bound}
\end{align}

Recall from \eqref{eq:Gn_recursion} that
$$
\Delta_{j+1} =
\frac{1}{j+1}\Bigl(-\frac12 G_j+H(G_j)+R_j\Bigr)
+\frac{U_{j+1}}{\sqrt{j+1}},
$$
with
$$
U_{j+1}
=
X_{j+1}-\frac{H(S_j)}{j^{(\alpha+1)/2}}
=
X_{j+1}-\frac{H(G_j)}{\sqrt j}.
$$
Write $\Delta_{j+1}=A_j+B_j$, where
$$
A_j:=
\frac{1}{j+1}\Bigl(-\frac12 G_j+H(G_j)+R_j\Bigr)
-\frac{H(G_j)}{\sqrt{j(j+1)}},
\quad
B_j:=\frac{X_{j+1}}{\sqrt{j+1}}.
$$
Since $\|H(x)\|=\|x\|^\alpha\le 1+\|x\|$ and
$\|R_j\|\le Cj^{-1}(\|G_j\|+\|G_j\|^\alpha)$, we have
$
\|A_j\|\le \frac{C}{j+1}(1+\|G_j\|).
$
Hence, by Lemma~\ref{lem:tightness},
\begin{align}\label{Aj.bound}
  \E\|A_j\|^2\le \frac{C}{(j+1)^2}.  
\end{align}
Moreover, by Assumption~\ref{assump},
$
\E\|B_j\|^2
=
\frac{1}{j+1}\E\|X_{j+1}\|^2
\le \frac{C}{j+1}.
$
Therefore
\begin{equation}\label{eq:Delta_second}
\E\|\Delta_{j+1}\|^2\le \frac{C}{j+1}.
\end{equation}

We claim that for every $\eta>0$,
\begin{equation}\label{eq:Delta_trunc}
\lim_{n\to\infty}\sum_{j=n}^{\infty}
\E\!\left[\|\Delta_{j+1}\|^2
\mathbf 1_{\{\|\Delta_{j+1}\|>\eta\}}\right]=0.
\end{equation}
Indeed, for every $\eta>0$,
$$
\|\Delta_{j+1}\|^2\mathbf 1_{\{\|\Delta_{j+1}\|>\eta\}}
\le
4\|A_j\|^2\mathbf 1_{\{\|A_j\|>\eta/4\}}
+
4\|B_j\|^2\mathbf 1_{\{\|B_j\|>\eta/4\}}.
$$
Using \eqref{Aj.bound}, we notice that
$$
\sum_{j=n}^{\infty}
\E\!\left[\|A_j\|^2\mathbf 1_{\{\|A_j\|>\eta/4\}}\right]
\le
\sum_{j=n}^{\infty}\E\|A_j\|^2
\le
C\sum_{j=n}^{\infty}(j+1)^{-2}\to0.
$$
On the other hand, by Assumption~\ref{assump},
$
\E\!\left[\|B_j\|^2\mathbf 1_{\{\|B_j\|>\eta/4\}}\right]
\le
C_\eta (j+1)^{-p/2}.
$
Since $p>2$, we have
$$
\sum_{j=n}^{\infty}
\E\!\left[\|B_j\|^2\mathbf 1_{\{\|B_j\|>\eta/4\}}\right]
\le
C_\eta\sum_{j=n}^{\infty}(j+1)^{-p/2}\to0.
$$
This proves \eqref{eq:Delta_trunc}.

Now fix $\eta>0$. By \eqref{eq:r_bound},
$
|r_{j+1}|
\le
\frac12\,\omega_\varphi(\eta)\,\|\Delta_{j+1}\|^2
+\|\nabla^2\varphi\|_\infty
\|\Delta_{j+1}\|^2\mathbf 1_{\{\|\Delta_{j+1}\|>\eta\}}.
$
Using \eqref{eq:Delta_second}, \eqref{eq:Delta_trunc}, and
$\sum_{j=n}^{m(n,T)-1}(j+1)^{-1}\le T$, we obtain
\begin{equation}\label{eq:r_sum}
\sum_{j=n}^{m(n,T)-1}\E|r_{j+1}|
\le
C\,T\,\omega_\varphi(\eta)+o_n(1).
\end{equation}
Letting first $n\to\infty$ and then $\eta\downarrow0$, we conclude that
\begin{equation}\label{eq:r_sum_vanish}
\lim_{n\to\infty}\sum_{j=n}^{m(n,T)-1}\E|r_{j+1}|=0.
\end{equation}

We next compute the conditional expectation of the first-order term in
\eqref{eq:taylor_phi}. Since $\E[U_{j+1}\mid\mathcal F_j]=0$,
$$
\E[\Delta_{j+1}\mid\mathcal F_j]
=
\frac{1}{j+1}\Bigl(-\frac12 G_j+H(G_j)+R_j\Bigr).
$$
Hence
$$
\E[\nabla\varphi(G_j)\cdot \Delta_{j+1}\mid\mathcal F_j]
=
\frac{1}{j+1}\nabla\varphi(G_j)\cdot\Bigl(-\frac12 G_j+H(G_j)\Bigr)
+
\frac{1}{j+1}\nabla\varphi(G_j)\cdot R_j.
$$
Since $\nabla\varphi$ is bounded and
$
\|R_j\|\le \frac{C}{j}\bigl(\|G_j\|+\|G_j\|^\alpha\bigr)$,
using Lemma~\ref{lem:tightness}, we obtain that
\begin{align}
\sum_{j=n}^{m(n,T)-1}
\E\left[
\left|
\frac{1}{j+1}\nabla\varphi(G_j)\cdot R_j
\right|
\right]
\le
C\sum_{j=n}^{\infty}\frac{1}{j(j+1)}
\Bigl(\E\|G_j\|+\E\|G_j\|^\alpha\Bigr)
\to0.
\label{eq:first_order_error}
\end{align}

We now treat the second-order term. By \eqref{eq:Gn_recursion},
$$
\Delta_{j+1}
=
a_j+b_j,
\quad\text{with}\quad
a_j:=\frac{1}{j+1}\Bigl(-\frac12 G_j+H(G_j)+R_j\Bigr)
\quad\text{and}\quad
b_j:=\frac{U_{j+1}}{\sqrt{j+1}}.
$$
Then
$\Delta_{j+1}\Delta_{j+1}^\top
=
a_j a_j^\top + a_j b_j^\top + b_j a_j^\top + b_j b_j^\top.$ Taking conditional expectation and using $\E[b_j\mid\mathcal F_j]=0$, we get
$$
\E[\Delta_{j+1}\Delta_{j+1}^\top\mid\mathcal F_j]
=
a_j a_j^\top + \E[b_j b_j^\top\mid\mathcal F_j].
$$
Note that
$$
\E[b_j b_j^\top\mid\mathcal F_j]
=
\frac{1}{j+1}\Sigma
-
\frac{1}{j(j+1)}H(G_j)H(G_j)^\top.
$$
Therefore
\begin{align*}
\E\left[\frac12\Delta_{j+1}^\top\nabla^2\varphi(G_j)\Delta_{j+1}\,\middle|\,\mathcal F_j\right]
&=
\frac{1}{2(j+1)}\tr\!\bigl(\Sigma\nabla^2\varphi(G_j)\bigr)
+\varepsilon_{j+1}^{(2)},
\end{align*}
where
$$
\varepsilon_{j+1}^{(2)}
=
-\frac{1}{2j(j+1)}
\tr\!\bigl(H(G_j)H(G_j)^\top \nabla^2\varphi(G_j)\bigr)
+
\frac12\tr\!\bigl(a_j a_j^\top\nabla^2\varphi(G_j)\bigr).
$$
Since $\nabla^2\varphi$ is bounded, $\|H(G_j)\|^2=\|G_j\|^{2\alpha}$, and
$\sup_j \E\|G_j\|^{2\alpha}<\infty$ by Lemma~\ref{lem:tightness},
we get
\begin{align}
\sum_{j=n}^{m(n,T)-1}\E|\varepsilon_{j+1}^{(2)}|
\le
C\sum_{j=n}^{\infty}\frac{1}{j^2}
+
C\sum_{j=n}^{\infty}\frac{1}{(j+1)^2}\E\Bigl\|-\frac12 G_j+H(G_j)+R_j\Bigr\|^2
\to0.
\label{eq:second_order_error}
\end{align}
Combining \eqref{eq:taylor_phi}, \eqref{eq:r_sum_vanish}, \eqref{eq:first_order_error},
and \eqref{eq:second_order_error}, we conclude that
$$
\E[\varphi(G_{j+1})-\varphi(G_j)\mid\mathcal F_j]
=
\frac{1}{j+1}\mathcal{L}\varphi(G_j)+\eta_{j+1},
$$
where
\begin{align*}
\lim_{n\to\infty}\sum_{j=n}^{m(n,T)-1}\E|\eta_{j+1}|=0.
\end{align*}
Consequently,
\begin{align} \label{eq:eta_sum}
    \E\left[
\sup_{0\le t\le T}
\left|
\sum_{j=n}^{m(n,t)-1}\eta_{j+1}
\right|
\right]
\le
\sum_{j=n}^{m(n,T)-1}\E|\eta_{j+1}|
\to0.
\end{align}

It remains to compare the Riemann sum with the integral term. Since
$\bar G^{(n)}(s)=G_j$ on $[t_j-t_n,\ t_{j+1}-t_n)$, we have
$$
\int_0^t \mathcal{L}\varphi(\bar G^{(n)}(s))\rmd s
=
\sum_{j=n}^{m(n,t)-1}\frac{1}{j+1}\mathcal{L}\varphi(G_j)+\rho_n(t),
$$
where
$$
|\rho_n(t)|
\le
\|\mathcal{L}\varphi\|_\infty \sup_{j\ge n}\frac{1}{j+1}
\le \frac{\|\mathcal{L}\varphi\|_\infty}{n+1}.
$$
Hence
\begin{align}
\lim_{n\to\infty}\sup_{0\le t\le T}|\rho_n(t)|=0.
\label{eq:Riemann_error}
\end{align}
Combining \eqref{eq:eta_sum} with \eqref{eq:Riemann_error}, we obtain the result of the lemma.
\end{proof}

We can now identify the weak limits of $(\bar G^{(n)})$.

\begin{lemma}\label{lem:critical_line_finite_horizon}
Fix $T>0$. Let $(n_k)$ be a subsequence such that $G_{n_k}$ converges weakly to $\nu$ for some probability measure $\nu$ on $\R^d$.
Then
$\bar G^{(n_k)}$ converges weakly to
$X^\nu$ in  $D([0,T],\R^d)$ as $k\to\infty$,
where $(X_t^\nu)_{0\le t\le T}$ is the unique weak solution of
\begin{equation}\label{eq:critical_line_limit_SDE}
\rmd X_t
=
\Bigl(-\frac12X_t+H(X_t)\Bigr)\rmd t+\Sigma^{1/2}\rmd B_t,
\quad X_0\sim \nu.
\end{equation}
\end{lemma}
\begin{proof}
By Lemma~\ref{Aldous.tightness}, the sequence $(\bar G^{(n_k)})$ is tight in
$D([0,T],\R^d)$.
Hence along a further subsequence (not relabeled) we may assume
$\bar G^{(n_k)}$ weakly converges to $Y$ in $D([0,T],\R^d)$ for some càdlàg process $Y$. 
By Lemma~\ref{lem:critical_line_small_jumps}, we have
$$
\sup_{0\le t\le T}\|\Delta \bar G^{(n_k)}(t)\|
\xrightarrow[k\to\infty]{\P}0.
$$
By Theorem~13.4 in \cite{BP1999}, applied componentwise after the linear time change from $[0,T]$ to $[0,1]$, every weak limit of $(\bar G^{(n_k)})$ has continuous paths almost surely. Hence $Y$ is almost surely continuous.

Let $\varphi\in C_c^2(\R^d)$ and define
$$
M_t^{(n)}(\varphi)
:=
\varphi(\bar G^{(n)}(t))
-
\varphi(\bar G^{(n)}(0))
-
\int_0^t \mathcal{L}\varphi(\bar G^{(n)}(s))\rmd s,
$$
where we recall that the operator $\mathcal{L}$ is given by \eqref{L.df}. By Lemma~\ref{lem:critical_line_generator},
\begin{align}\label{L1.conv}
   \sup_{0\le r\le T}
\left|
\sum_{j=n}^{m(n,r)-1}
\E[\varphi(G_{j+1})-\varphi(G_j)\mid\mathcal F_j]
-
\int_0^r \mathcal{L}\varphi(\bar G^{(n)}(u))\rmd u
\right|
\to0
\quad\text{in }L^1. 
\end{align}
Let $0\le s\le t\le T$ and let $\Psi:D([0,s],\R^d)\to\R$ be a bounded continuous functional. Since
$\bar G^{(n)}|_{[0,s]}$ is $\mathcal F_{m(n,s)}$-measurable and
$$
D_{j+1}:=\varphi(G_{j+1})-\varphi(G_j)
-\E[\varphi(G_{j+1})-\varphi(G_j)\mid\mathcal F_j]
$$
is a martingale difference sequence, we have
$$
\E\left[
\sum_{j=m(n,s)}^{m(n,t)-1}D_{j+1}
\,
\Psi\bigl(\bar G^{(n)}|_{[0,s]}\bigr)
\right]=0.
$$
Combining this identity with \eqref{L1.conv}, we obtain
\begin{equation}\label{eq:approx_mart}
\E\!\left[
\bigl(M_t^{(n)}(\varphi)-M_s^{(n)}(\varphi)\bigr)
\Psi\!\bigl(\bar G^{(n)}|_{[0,s]}\bigr)
\right]
\to 0.
\end{equation}

We now pass to the limit along $(n_k)$. Since $Y$ has continuous paths almost
surely, the maps
$$
\omega\mapsto \omega(t),\quad
\omega\mapsto \omega(s),\quad
\omega\mapsto \int_s^t \mathcal L\varphi(\omega(u))\rmd u
$$
are continuous at $Y$-almost every path in the $J_1$ topology on $D([0,T],\R^d)$. Therefore the functional
$$
\omega\mapsto
\Bigl(
\varphi(\omega(t))-\varphi(\omega(s))
-\int_s^t \mathcal L\varphi(\omega(u))\rmd u
\Bigr)\Psi(\omega|_{[0,s]})
$$
is $Y$-almost surely continuous and bounded. Therefore, by weak convergence of
$\bar G^{(n_k)}$ to $Y$ and \eqref{eq:approx_mart}, we have
$$
\E\left[
\Bigl(\varphi(Y_t)-\varphi(Y_s)-\int_s^t \mathcal{L}\varphi(Y_u)\rmd u\Bigr)
\Psi(Y|_{[0,s]})
\right]
=0.
$$
Since this holds for every bounded continuous $\Psi$, it follows that
$$
\varphi(Y_t)-\varphi(Y_0)-\int_0^t \mathcal L\varphi(Y_u)\rmd u
$$
is a martingale for every $\varphi\in C_c^2(\R^d)$. Hence $Y$ solves the
martingale problem for $\mathcal L$.

By Lemma~\ref{sde.unique_solution}, the SDE \eqref{eq:critical_line_limit_SDE}
is well posed in the weak sense for every initial law. Hence the martingale
problem for $\mathcal{L}$ is well posed (see, e.g., Chapter~4 and Chapter~7
in \cite{EK1986}).
Since $\bar G^{(n_k)}(0)=G_{n_k}$ converges in distribution to $\nu$, the initial law of $Y$ is $\nu$.
Therefore $Y$ has the same law as $X^\nu$, the unique weak solution of
\eqref{eq:critical_line_limit_SDE} with initial law $\nu$. Since every convergent subsequence has the same limit law, the whole sequence
$\bar G^{(n_k)}$ converges weakly to $X^\nu$ in $D([0,T],\R^d)$.
\end{proof}

We can now prove the convergence in distribution of $G_n$.

\begin{proof}[Proof of Theorem~\ref{thm.nonlinear}.b]
Let $(G_{n_k})_{k\ge 1}$ be an arbitrary subsequence of $(G_n)_{n\ge 1}$. By Lemma~\ref{lem:tightness}, the
sequence $(G_n)_{n\ge1}$ is tight in $\R^d$. Hence, after passing to a further
subsequence if necessary, we may assume that
$(G_{n_k})$ converges in distribution to a probability measure $\nu$ on $\R^d$.
Fix $T>0$ and define
$$
\ell_k:=\min\{m\le n_k:\ t_{n_k}-t_m\le T\}.
$$
Then $\ell_k\to\infty$. Let
$
T_k:=t_{n_k}-t_{\ell_k}$. We notice that $T_k\to T$. Indeed, for all sufficiently large $k$, the minimality of $\ell_k$ implies
$
t_{n_k}-t_{\ell_k}\le T < t_{n_k}-t_{\ell_k-1},
$
hence
$$
0\le T-T_k < t_{\ell_k}-t_{\ell_k-1}=\frac{1}{\ell_k}\to 0 \quad\text{as $k\to\infty$}.
$$

For this fixed $T$, by tightness, after passing to a further subsequence, not relabeled, we may assume
$G_{\ell_k}$ converges in distribution to a probability measure $\nu_T$ on $\R^d$. Now consider the process $\bar G^{(\ell_k)}$ on the interval $[0,T+1]$.
Since $\bar G^{(\ell_k)}(0)=G_{\ell_k}$ converges in distribution to $\nu_T$, by Lemma~\ref{lem:critical_line_finite_horizon}, we infer that
$\bar G^{(\ell_k)}$ converges weakly to
$X^{\nu_T}$ in $D([0,T+1],\R^d)$,
where $X^{\nu_T}$ is the unique weak solution of \eqref{eq:sde_critical}
with initial law $\nu_T$.

By definition of $\bar G^{(\ell_k)}$, we have
$
\bar G^{(\ell_k)}(t)=G_{m(\ell_k,t)}.
$
Since $T_k=t_{n_k}-t_{\ell_k}$, it follows that
$
m(\ell_k,T_k)=n_k,
$
and thus
$
\bar G^{(\ell_k)}(T_k)=G_{n_k}.
$
Since $T_k\to T$ deterministically and $\bar G^{(\ell_k)}$ converges weakly to
$X^{\nu_T}$ in $D([0,T+1],\R^d)$, the pair
$(\bar G^{(\ell_k)},T_k)$ converges weakly to $(X^{\nu_T},T)$ in
$D([0,T+1],\R^d)\times[0,T+1]$. Consider the map
$\Phi:D([0,T+1],\R^d)\times[0,T+1]\to\R^d$ defined by
$\Phi(\omega,r)=\omega(r)$. This map is continuous at every
$(\omega,r)$ such that $\omega$ is continuous at $r$. Since
$X^{\nu_T}$ has continuous paths almost surely,  the continuous mapping theorem (see Theorem~2.7 in
\cite{BP1999}) yields that
$\bar G^{(\ell_k)}(T_k)$ converges in distribution to $X_T^{\nu_T}$.
Therefore $G_{n_k}$ converges in distribution to $X_T^{\nu_T}$.
As $G_{n_k}$ converges in distribution to $\nu$, we conclude that
\begin{align}
   \nu = \nu_T P_T,
\label{eq:nu_PT} 
\end{align}
where $(P_t)_{t\ge0}$ is the semigroup of the limiting diffusion.

Let $\pi$ be the unique invariant probability measure of the diffusion \eqref{eq:sde_critical}. By \eqref{eq:nu_PT}, for every bounded continuous function $f:\R^d\to\R$, we have
$$
\nu(f)-\pi(f)=\nu_T(P_Tf-\pi(f)).
$$
By Lemma~\ref{sde.unique_solution}, there exist constants
$C,\lambda>0$ such that
$$
|\nu(f)-\pi(f)|
\le
C e^{-\lambda T}\|f\|_\infty \int_{\R^d}(1+\|x\|^2)\,\nu_T(\rmd x).
$$
It remains to bound the second moment under $\nu_T$. By the Portmanteau theorem (see Theorem~2.1 in \cite{BP1999}), since $x\mapsto 1+\|x\|^2$ is nonnegative and lower semicontinuous, and since $\nu_T$ is a weak limit of $(G_{\ell_k})$, we have
$$
\int_{\R^d}(1+\|x\|^2)\,\nu_T(\rmd x)
\le
\liminf_{k\to\infty}\E[1+\|G_{\ell_k}\|^2]
\le
1+\sup_{m\ge1}\E[\|G_m\|^2]
<\infty,
$$
where in the last step, we used Lemma~\ref{lem:tightness}. Therefore there exists a constant $C'>0$, independent of $T$, such that
$$
|\nu(f)-\pi(f)|
\le
C' e^{-\lambda T}\|f\|_\infty.
$$
Taking $T\to\infty$, we obtain
$
\nu(f)=\pi(f)
$
for every bounded continuous function $f$.
Hence $\nu=\pi$. Since $(G_n)_{n\ge 1}$ is tight by Lemma~\ref{lem:tightness} and every subsequential weak limit of $(G_n)_{n\ge 1}$ is equal to $\pi$, we conclude that
$G_n=\frac{S_n}{\sqrt n}$ converges in distribution to $\pi$.
For $d=1$, let $p$ be the probability density function defined in Theorem~\ref{thm.nonlinear}(b). Since $\alpha+1<2$, its normalizing constant is finite. Notice that
$$
\frac{\Sigma}{2}p'(x)=\left(-\frac x2+H(x)\right)p(x).
$$
Thus $p$ is the stationary density of the one-dimensional diffusion. By uniqueness of the stationary distribution, $\pi(\rmd x)=p(x)\rmd x$. This completes the proof.
\end{proof}

\subsection{Under the critical line}

Let $\gamma = (1-\beta)/(1-\alpha)$. As $2\beta-1<\alpha\le \beta$ and $0\le \beta<1$, we have $\gamma\in (1/2,1]$.
Set $G_n=n^{-\gamma} S_n$ and $F(x):=-\gamma x+H(x)$. Let $U_1=X_1$ and for $n\ge 1$, let $$U_{n+1}:=X_{n+1}-\E[X_{n+1}\mid\mathcal F_n]=X_{n+1}-\frac{\|S_n\|^{\alpha-1}S_n}{n^{\beta}}=X_{n+1}-\frac{\|G_n\|^{\alpha-1}G_n}{n^{1-\gamma}}.$$
Using the fact that $S_{n+1}=S_n+X_{n+1}$, we obtain
\begin{equation}\label{recur.eq}
G_{n+1}= (n+1)^{-\gamma}S_n+(n+1)^{-\gamma}X_{n+1}
      =\Bigl(\frac{n}{n+1}\Bigr)^{\gamma}G_n
       +\frac{1}{n^{1-\gamma}(n+1)^{\gamma}}\|G_n\|^{\alpha-1}G_n
       +\frac{U_{n+1}}{(n+1)^{\gamma}}.
\end{equation}
Using the Taylor expansions $$
\left(\frac{n}{n+1}\right)^\gamma
=
1-\frac{\gamma}{n+1}+O((n+1)^{-2}),\quad \frac{1}{n^{1-\gamma}(n+1)^\gamma}
=
\frac{1}{n+1}
+O\big((n+1)^{-2}\big).
$$
and substituting into \eqref{recur.eq}, we get
\begin{align}\label{G.SAeq}
    G_{n+1}=G_n+\frac{1}{n+1}\bigl(-\gamma G_n+\|G_n\|^{\alpha-1}G_n+R_n\bigr) +\frac{U_{n+1}}{(n+1)^{\gamma}}
\end{align}
where 
\begin{align}\label{inq.Rn}
    \|R_n\|\le \frac{C}{n}(\|G_n\|+\|G_n\|^{\alpha})
\end{align}
for some deterministic constant $C>0$.

In the next lemma, we give an upper bound for $G_n$.

\begin{lemma}\label{lem:S.bound}
  For each $\eps>0$, there exists a.s. a random positive constant $C_{\eps}$ such that $$\|G_n\| \le C_{\eps} n^{\eps}, \quad  \forall n\ge 1.$$
\end{lemma}
\begin{proof}
	Let $Z_n :=  n^{-2\eps}\|G_n\|^2= n^{-2\gamma-2\eps} \|S_n\|^2$. We apply Robbins–Siegmund theorem
	(see Proposition~\ref{RS.thm} in the appendix) to show that $Z_n$ is bounded almost surely. As in the proof of Lemma~\ref{Sb.lem}, we have
	\begin{equation}\label{eq:Zrec_under}
		\mathbb{E}[Z_{n+1}\mid\mathcal{F}_n]
		= \frac{n^{2\gamma+2\eps}}{(n+1)^{2\gamma+2\eps}}\,Z_n
		+ \frac{2\|S_n\|^{\alpha+1}}{n^\beta(n+1)^{2\gamma+2\eps}}
		+ \frac{\mathrm{tr}(\Sigma)}{(n+1)^{2\gamma+2\eps}}.
	\end{equation}
	Using the fact that $\|S_n\|^2 = n^{2(\gamma+\eps)}Z_n$
	and $\gamma(1-\alpha)=1-\beta$, we have 
	\begin{equation}
	    \label{eq:Young_Zn_under}
        \frac{\|S_n\|^{\alpha+1}}{n^\beta(n+1)^{2(\gamma+\eps)}}\le\frac{Z_n^{(\alpha+1)/2}}{n^{\beta+(\gamma +\eps)(1 -\alpha)}}
	= \frac{Z_n^{(\alpha+1)/2}}{n^{1+\eps(1-\alpha)}}\le \frac{\gamma+\eps}{2n} Z_n+\frac{C_{\alpha}}{n^{1+2\eps}},
	\end{equation}
	with some constant $C_{\alpha}>0$ depending only on $\alpha$, where the last bound is obtained by applying the Young inequality
	$AB \le \frac{A^p}{p}+\frac{B^q}{q}$ with
	$p=\frac{2}{\alpha+1}$, $q=\frac{2}{1-\alpha}$,
	$A = \mu Z_n^{(\alpha+1)/2}$,
	$B = \mu^{-1} n^{-[1+\eps(1-\alpha)]}$,
	and choosing $\mu^{2/(\alpha+1)} = \frac{\gamma+\eps}{2n}$.	
	Substituting \eqref{eq:Young_Zn_under} into \eqref{eq:Zrec_under}
	and using $\frac{n^{2\gamma+2\eps}}{(n+1)^{2\gamma+2\eps}}
	= 1 - \frac{2\gamma+2\eps}{n+1} + O(n^{-2})$, we obtain that
	$$
	\mathbb{E}[Z_{n+1}\mid\mathcal{F}_n]
	\le (1+\alpha_n)\,Z_n + \beta_n - Y_n,
	$$
	where
	$$
	\alpha_n = O(n^{-2}), \quad
	\beta_n = \frac{2C_\alpha}{n^{1+2\eps}}
	+ \frac{\mathrm{tr}(\Sigma)}{(n+1)^{2(\gamma+\eps)}},
	\quad
	Y_n = \frac{(\gamma+\eps)(n-1)}{n(n+1)}
Z_n.
	$$
	It is clear that all the conditions of Proposition~\ref{RS.thm} are fulfilled: $\sum_n\alpha_n<\infty$, $\sum_n\beta_n<\infty$  (as $\gamma>1/2$) and 
  $Y_n\ge 0$ for all $n\ge 1$. By Proposition~\ref{RS.thm}, $Z_n$ converges a.s. to a finite limit, hence $Z_n=O(1)$ a.s., which proves the lemma.
\end{proof}

Fix a sufficiently small $\eps\in (0,1)$ and a positive integer $m$. Let
$$\sigma_{m,\eps}:=\inf\{n\ge 1:\ \|G_n\|\,n^{-\eps}\ge m\}.$$
By Lemma \ref{lem:S.bound}, we have
$$\P(\bigcup_{m\ge1} \{\sigma_{m,\eps}=\infty\})=1.$$
Let $$V_n:=\|G_n\|^2 \text{ for }n\ge 1 \quad \text{and}\quad h(v):=2v^{(\alpha+1)/2}-2\gamma v \text{ for }v\ge0.$$

\begin{lemma}\label{prop:pathwise_recursion}
Fix $m\in \N$ and $0<\varepsilon<\min\{1/2,\gamma-1/2\}$. For each $n\ge 1$, we have
\begin{equation}\label{eq:pathwise_recursion_Omega}
V_n
=
V_1+\sum_{k=1}^{n-1}\frac1{k+1} h(V_k)+Q_{n},
\end{equation}
where $(Q_n)_{n\ge 1}$ is a $\mathcal{F}_n$-adapted process such that $Q_{n\wedge \sigma_{m,\eps}}$ converges almost surely as $n\to\infty$.
\end{lemma}

\begin{proof}
 From \eqref{G.SAeq}, we have
\begin{align}\label{eq:V_recursion}
     V_{n+1}
=
V_n+\frac1{n+1} h(V_n)+\vartheta_{n+1}+\rho_{n+1},
\end{align}
where
$$
\vartheta_{n+1}
:=
\frac2{(n+1)^{\gamma}}\Bigl\langle G_n+\frac1{n+1}(F(G_n)+R_n),U_{n+1}\Bigr\rangle,
$$
which is a martingale difference, and 
$$
\rho_{n+1}
:=
\frac{2}{n+1}\langle G_n,R_n\rangle
+\frac1{(n+1)^2}\|F(G_n)+R_n\|^2
+\frac{1}{(n+1)^{2\gamma}}\|U_{n+1}\|^2.
$$
Let
$$
M_n:=\sum_{k=1}^{n-1}\vartheta_{k+1},
\quad
A_n:=\sum_{k=1}^{n-1}\rho_{k+1},
\quad
Q_n:=M_n+A_n.
$$
Then \eqref{eq:pathwise_recursion_Omega} follows immediately from
\eqref{eq:V_recursion}. It remains to show that
$A_{n\wedge\sigma_{m,\eps}}$ and $M_{n\wedge\sigma_{m,\eps}}$ converge a.s.

We first show the almost sure convergence of $A_{n\wedge\sigma_{m,\eps}}$.
On $\{\sigma_{m,\eps}<\infty\}$, the sequence
$A_{n\wedge\sigma_{m,\eps}}$ is eventually constant, hence converges trivially. It therefore suffices to work on
$\{\sigma_{m,\eps}=\infty\}$, where
$\|G_n\|<mn^\eps$ holds for all $n\ge1$ by definition of
$\sigma_{m,\eps}$. On this event, since $\|x\|^\alpha\le 1+\|x\|$ for
$\alpha\in(0,1)$, we have
$$
\|F(G_n)\|\le \gamma\|G_n\|+\|G_n\|^\alpha
\le C_m n^\eps,
$$
and by \eqref{inq.Rn},
$$
\|R_n\|\le \frac{C}{n}(\|G_n\|+\|G_n\|^\alpha)
\le C_m n^{-1+\eps}.
$$
Therefore,
$$
\left|
\frac{2}{n+1}\langle G_n,R_n\rangle
\right|
\le
\frac{C_m}{n^{2-2\eps}} \quad\text{and}\quad\frac1{(n+1)^2}\|F(G_n)+R_n\|^2
\le
\frac{C_m}{n^{2-2\eps}}.
$$
Since $\eps<1/2$, these two series are summable on
$\{\sigma_{m,\eps}=\infty\}$.

It remains to control the last term in $\rho_{n+1}$. Since
$
\E[\|U_{n+1}\|^2\mid\mathcal F_n]\le \mathrm{tr}(\Sigma),
$
we have
$$
\E\left[
\sum_{n=1}^\infty
\frac{\|U_{n+1}\|^2}{(n+1)^{2\gamma}}
\right]
=
\sum_{n=1}^\infty
\frac{\E[\|U_{n+1}\|^2]}{(n+1)^{2\gamma}}
\le
\mathrm{tr}(\Sigma)
\sum_{n=1}^\infty \frac1{(n+1)^{2\gamma}}
<\infty,
$$
as $\gamma>1/2$. Hence
$$
\sum_{n=1}^\infty
\frac{\|U_{n+1}\|^2}{(n+1)^{2\gamma}}
<\infty
\quad\text{a.s.}
$$
Thus $\sum_n|\rho_{n+1}|<\infty$ a.s. on $\{\sigma_{m,\eps}=\infty\}$, and
therefore $A_{n\wedge\sigma_{m,\eps}}$ converges almost surely.

We now prove the almost sure convergence of $M_{n\wedge\sigma_{m,\eps}}$.
Note that $(M_{n\wedge\sigma_{m,\eps}})_{n\ge1}$ is a martingale. Moreover, since
$\{k<\sigma_{m,\eps}\}\in\mathcal F_k$, we have
\begin{align*}
\E\left[
\vartheta_{k+1}^2\mathbf 1_{\{k<\sigma_{m,\eps}\}}
\mid\mathcal F_k
\right]
&\le
\frac{C}{(k+1)^{2\gamma}}
\mathbf 1_{\{k<\sigma_{m,\eps}\}}
\left\|G_k+\frac1{k+1}(F(G_k)+R_k)\right\|^2
\E[\|U_{k+1}\|^2\mid\mathcal F_k].
\end{align*}
On $\{k<\sigma_{m,\eps}\}$, we have
$\|G_k\|<mk^\eps$,
$\|F(G_k)\|\le C_m k^\eps$, and
$\|R_k\|\le C_m k^{-1+\eps}$. Hence
$$
\E\left[
\vartheta_{k+1}^2\mathbf 1_{\{k<\sigma_{m,\eps}\}}
\mid\mathcal F_k
\right]
\le
\frac{C_m}{k^{2\gamma-2\eps}}.
$$
Since $\eps<\gamma-\frac12$, we have
$
\sum_{k=1}^\infty \frac{C_m}{k^{2\gamma-2\eps}}<\infty.
$
Consequently,
$$
\sup_{n\ge1}\E\left[M_{n\wedge\sigma_{m,\eps}}^2\right]
=
\sup_{n\ge1}
\sum_{k=1}^{n-1}
\E\left[
\vartheta_{k+1}^2\mathbf 1_{\{k<\sigma_{m,\eps}\}}
\right]
<\infty.
$$
Thus $(M_{n\wedge\sigma_{m,\eps}})_{n\ge1}$ is bounded in $L^2$, and by the
$L^2$ martingale convergence theorem, $M_{n\wedge\sigma_{m,\eps}}$
converges almost surely. Combining the almost sure convergence of $A_{n\wedge\sigma_{m,\eps}}$ and
$M_{n\wedge\sigma_{m,\eps}}$, we conclude that
$Q_{n\wedge\sigma_{m,\eps}}$ converges almost surely.
This completes the proof.
\end{proof}

The next lemma shows that the jumps of $(V_n)$ vanish on each localization event $\{\sigma_{m,\eps}=\infty\}$.

\begin{lemma}\label{lem:small_jumps}
Fix $m\in \N$ and $0<\varepsilon<\min\{1/2,\gamma-1/2\}$. On the event $\{\sigma_{m,\eps}=\infty\}$, we have that a.s.,
$V_{n+1}-V_n\to 0$ as $n\to\infty$.
\end{lemma}
\begin{proof}
Fix $m\ge1$ and work on the event $\{\sigma_{m,\eps}=\infty\}$.
Recall from Lemma~\ref{prop:pathwise_recursion} that on
$\{\sigma_{m,\eps}=\infty\}$, $Q_n$ converges almost surely. Hence
$Q_{n+1}-Q_n\to0$ almost surely.

It remains to control the drift term. On $\{\sigma_{m,\eps}=\infty\}$ we have
$
\|G_n\|<m n^\eps
$ for all $n\ge1$, and therefore
$
V_n\le m^2 n^{2\eps}.
$
Recall that
$
h(v)=2v^{(\alpha+1)/2}-2\gamma v.
$
Since $(\alpha+1)/2\le1$, there exists a constant $C>0$ such that
$
|h(v)|\le C(1+v)
$
for all $v\ge0$. Therefore, on $\{\sigma_{m,\eps}=\infty\}$,
$$
\left|\frac1{n+1}h(V_n)\right|
\le
\frac{C}{n+1}(1+V_n)
\le
C_m n^{-1+2\eps}\to0,
$$
as $\eps<1/2$. Subtracting \eqref{eq:pathwise_recursion_Omega} at times $n+1$ and $n$, we have
$$
V_{n+1}-V_n
=
\frac1{n+1}h(V_n)+Q_{n+1}-Q_n.
$$
Combining the two preceding limits, we obtain that
$
V_{n+1}-V_n\to0$ a.s. on $\{\sigma_{m,\eps}=\infty\}.$
\end{proof}
The following result is in the spirit of the nonattracting-point technique used by Pemantle (see Chapter 3 in Pemantle's PhD thesis \cite{P1988}) to study the non-convergence of urn processes and stochastic approximations.

\begin{lemma}\label{lem:one_sided_exit}
Fix $m\ge1$, $0<\varepsilon<\min\{1/2,\gamma-1/2\}$, and let $I=[a,b]\subset(0,\infty)$. Assume that either $\inf_{v\in I}h(v)>0$ or $\sup_{v\in I}h(v)<0$.
Then for every interval $I'=[a',b']$ such that
$a<a'<b'<b$, the process $(V_n)_{n\ge 1}$ a.s. visits $I'$ only  finitely many times on the event $\{\sigma_{m,\eps}=\infty\}$.
\end{lemma}
\begin{proof}
We consider the case $\inf_{v\in I}h(v)>0$. Fix $I'=[a',b']$ with
$a<a'<b'<b$. Set
$$
\delta:=\operatorname{dist}(I',I^c)=\min\{a'-a, b-b'\}>0.
$$
Recall that by Lemma~\ref{prop:pathwise_recursion} and
Lemma~\ref{lem:small_jumps}, on $\{\sigma_{m,\eps}=\infty\}$, the sequence
$Q_n$ converges and $V_{n+1}-V_n\to0$ a.s. as $n\to\infty$. Therefore, on
$\{\sigma_{m,\eps}=\infty\}$, there exists a finite random integer $N$ such
that
\begin{align}
|Q_r-Q_s|&<\frac \delta4\quad\text{for all $r,s\ge N$,}
\label{eq:Cauchy_C}
\\
|V_{k+1}-V_k|&<\frac \delta4\quad\text{for all }k\ge N.
\label{eq:small_jumps_bound}
\end{align}
Fix $n\ge N$ with $V_n\in I'$. Define
$$
\tau:=\inf\{k\ge n:\ V_k\notin I\}.
$$
We first show that $\tau<\infty$. Suppose for contradiction that
$\tau=\infty$, so $V_k\in I$ for all $k\ge n$. Then by
\eqref{eq:pathwise_recursion_Omega},
$$
V_\ell-V_n
=
\sum_{k=n}^{\ell-1}\frac1{k+1}h(V_k)
+
Q_\ell-Q_n
\quad\text{for all }\ell>n.
$$
Since $h(V_k)\ge c_I:=\inf_{v\in I}h(v)>0$ for all $k\ge n$, we have
$$
V_\ell-V_n
\ge
c_I\sum_{k=n}^{\ell-1}\frac1{k+1}
-
|Q_\ell-Q_n|.
$$
By \eqref{eq:Cauchy_C}, we get
$$
V_\ell-V_n
\ge
c_I\sum_{k=n}^{\ell-1}\frac1{k+1}-\frac\delta4.
$$
Since the harmonic sum diverges, the right-hand side tends to $+\infty$ as
$\ell\to\infty$, which is impossible on the event $\{V_\ell\in I \text{ for all } \ell\ge n\}$. Thus $\tau<\infty$.

We next show that the exit is through the upper endpoint. Since $V_n\in I'$
and $V_\tau\notin I$, we have $\tau\ge n+1$ and $V_{\tau-1}\in I$. Using
again \eqref{eq:pathwise_recursion_Omega}, we have
$$
V_\tau-V_n
=
\sum_{k=n}^{\tau-1}\frac1{k+1}h(V_k)
+
Q_\tau-Q_n.
$$
Since $V_k\in I$ for all $n\le k\le\tau-1$, we have $h(V_k)\ge c_I>0$,
and therefore
$$
V_\tau-V_n
\ge
Q_\tau-Q_n
>
-\frac\delta4
$$
by \eqref{eq:Cauchy_C}. If $V_\tau<a$, then, since $V_n\in I'=[a',b']$,
$$
V_\tau-V_n\le a-a'\le-\delta,
$$
which is a contradiction. Hence $V_\tau>b$.

We now prove that the process cannot return to $I'$ after time $\tau$.
Assume, for contradiction, that there exists $s>\tau$ such that $V_s\in I'$.
Since $I'\subset I$, the first re-entry time
$$
m_1:=\inf\{k>\tau:\ V_k\in I\}
$$
is finite and satisfies $m_1\le s$. Since $V_\tau>b$, and since the jumps
are smaller than $\delta/4$ after time $N$, the first re-entry into $I$ must
occur from above; otherwise the process would have to jump from above $b$ to
below $a$ without visiting $I$, which is impossible. Hence
\begin{equation}\label{eq:first_reentry_one_sided}
V_{m_1}>b-\frac\delta4.
\end{equation}
Let
$$
\rho_1:=\inf\{k\ge m_1:\ V_k\notin I\}
$$
be the next exit time from $I$, with the convention that $\rho_1=\infty$ if
there is no such exit. For any finite $k$ such that $m_1\le k<\rho_1$, we
have $V_j\in I$ for all $m_1\le j\le k$. Thus, by
\eqref{eq:pathwise_recursion_Omega},
$$
V_k-V_{m_1}
=
\sum_{j=m_1}^{k-1}\frac1{j+1}h(V_j)
+
Q_k-Q_{m_1}
\ge
-\frac\delta4,
$$
where we used \eqref{eq:Cauchy_C} and the positivity of $h$ on $I$.
Combining this with \eqref{eq:first_reentry_one_sided}, we obtain
$
V_k>b-\delta/2>b'
$
for all finite $k$ with $m_1\le k<\rho_1$. Hence this whole excursion inside
$I$ cannot intersect $I'$.

If $\rho_1<\infty$, then the exit at time $\rho_1$ must also occur through
the upper endpoint. Indeed, since $V_{\rho_1-1}>b'$ and the jumps are smaller
than $\delta/4$, the process cannot jump below $a$ at time $\rho_1$. Therefore
$V_{\rho_1}>b$. If the process later re-enters $I$, define inductively
$$
m_{j+1}:=\inf\{k>\rho_j:\ V_k\in I\},
\quad
\rho_{j+1}:=\inf\{k\ge m_{j+1}:\ V_k\notin I\}.
$$
The same argument shows that each such re-entry occurs from above, each
excursion inside $I$ stays above $b'$, and each finite exit occurs again
through the upper endpoint. Therefore no later excursion into $I$ can
intersect $I'$. This contradicts the existence of $s>\tau$ with $V_s\in I'$. Hence, for every $n\ge N$ with $V_n\in I'$, after the corresponding exit time
$\tau$ the process never returns to $I'$. Since only finitely many times occur
before $\tau$, it follows that $(V_n)$ visits $I'$ only finitely many times on
$\{\sigma_{m,\eps}=\infty\}$.

The case $\sup_{v\in I}h(v)<0$ is obtained similarly by reversing the
inequalities.
\end{proof}

We next identify limit points of the sequence $(V_n)_{n\ge 1}$.
 
\begin{lemma}
\label{prop:limit_points}
Fix $m\ge1$ and $0<\varepsilon<\min\{1/2,\gamma-1/2\}$. On the event $\{\sigma_{m,\eps}=\infty\}$, the sequence $(V_n)_{n\ge 1}$ is bounded and every subsequential limit of $(V_n)_{n\ge 1}$ belongs to the set $\{0,r^2\}$.
\end{lemma}

\begin{proof}
Recall that
$
h(v)=2v^{(\alpha+1)/2}-2\gamma v
=
2v\bigl(v^{(\alpha-1)/2}-\gamma\bigr)$ for 
$v>0$. 
Since $0<\alpha<1$, the function $h$ satisfies
$$
h(v)>0 \quad\text{for }v\in(0,r^2) \quad\text{and}\quad
h(v)<0 \quad\text{for }v\in(r^2,\infty),
$$
where $r^{\alpha-1}=\gamma$. Moreover,
$h(v)\to-\infty$ as $v\to\infty$. Hence we can choose
$R>r^2+1$ sufficiently large such that
\begin{equation}\label{eq:h_negative_large}
h(v)\le -1
\quad\text{for all }v\ge R.
\end{equation}

We first show that on $\{\sigma_{m,\eps}=\infty\}$, the process cannot remain
forever above the level $R$. Indeed, suppose that there exists a finite random
time $N_0$ such that $V_n\ge R$ for all $n\ge N_0$. Then for every
$\ell>n\ge N_0$, using \eqref{eq:pathwise_recursion_Omega} and
\eqref{eq:h_negative_large}, we have
$$
V_\ell-V_n
=
\sum_{k=n}^{\ell-1}\frac{1}{k+1}h(V_k)
+
Q_\ell-Q_n
\le
-\sum_{k=n}^{\ell-1}\frac{1}{k+1}
+
Q_\ell-Q_n.
$$
On $\{\sigma_{m,\eps}=\infty\}$, the sequence $(Q_n)$ converges almost surely.
Since the harmonic sum diverges, the right-hand side tends to $-\infty$ as
$\ell\to\infty$, which contradicts $V_\ell\ge0$. Thus, on
$\{\sigma_{m,\eps}=\infty\}$, the process cannot stay in $[R,\infty)$ forever.

We next show that $(V_n)$ is bounded on $\{\sigma_{m,\eps}=\infty\}$. Suppose,
for contradiction, that $(V_n)$ is unbounded on this event. Since the process
cannot eventually remain in $[R,\infty)$, it visits both $[0,R)$ and
$(2R,\infty)$ infinitely often. By Lemma~\ref{lem:small_jumps}, there exists
a finite random time after which
$$
|V_{n+1}-V_n|<\frac R4.
$$
Therefore every sufficiently late passage from below $R$ to above $2R$, or
from above $2R$ to below $R$, must cross the interval
$
K:=\left[{5R}/{4},{7R}/{4}\right]\subset[R,2R].
$
Hence $K$ is visited infinitely often. However, by \eqref{eq:h_negative_large},
we have
$
\sup_{v\in[R,2R]}h(v)<0.
$
Applying Lemma~\ref{lem:one_sided_exit} with
$I=[R,2R]$ and $I'=K$, we conclude that $K$ can be visited only finitely many
times on $\{\sigma_{m,\eps}=\infty\}$. This contradiction proves that
$(V_n)$ is bounded on $\{\sigma_{m,\eps}=\infty\}$.

We apply Lemma~\ref{lem:one_sided_exit} simultaneously to the countable collection of pairs of nested intervals with rational endpoints on which $h$ has a strict sign. Outside one null set, every point of $(0,\infty)\setminus\{r^2\}$ lies in the interior of one of the smaller intervals. We now identify the possible subsequential limits. First, we show that there
are no subsequential limits in $(0,r^2)$. Let $L\in(0,r^2)$. Since $h$ is
continuous and strictly positive at $L$, there exists an interval
$I=[a,b]\subset(0,r^2)$ containing $L$ in its interior such that
$
\inf_{v\in I}h(v)>0.
$
Choose a smaller interval $I'=[a',b']$ such that
$
a<a'<L<b'<b.
$
By Lemma~\ref{lem:one_sided_exit}, the process $(V_n)$ visits $I'$ only
finitely many times on $\{\sigma_{m,\eps}=\infty\}$. This contradicts the
possibility that $L$ is a subsequential limit. Therefore no subsequential
limit can lie in $(0,r^2)$.

Similarly, if $L\in(r^2,\infty)$, then by continuity of $h$ and the fact that
$h(v)<0$ on $(r^2,\infty)$, we can choose an interval
$I=[a,b]\subset(r^2,\infty)$ containing $L$ in its interior such that
$
\sup_{v\in I}h(v)<0.
$
Choosing a smaller interval $I'=[a',b']$ with
$a<a'<L<b'<b$ and applying Lemma~\ref{lem:one_sided_exit}, we conclude that
$L$ cannot be a subsequential limit.

Thus, on $\{\sigma_{m,\eps}=\infty\}$, the sequence $(V_n)$ is bounded and
every subsequential limit belongs to $\{0,r^2\}$.
\end{proof}
We can now show that $(V_n)$ converges almost surely.

\begin{proposition}\label{lem:dist}
We have
$$
\operatorname{dist}\bigl(n^{-\gamma}S_n,\{0\}\cup \mathcal C_r\bigr)\to0
\quad\text{a.s.}
$$
\end{proposition}

\begin{proof}
Fix $0<\eps<\min\{1/2,\gamma-1/2\}$ and $m\ge1$, and work on the event
$\{\sigma_{m,\eps}=\infty\}$. By Lemma~\ref{prop:limit_points}, the sequence
$(V_n)$ is bounded and all of its subsequential limits belong to the set
$\{0,r^2\}$.

We claim that $(V_n)$ cannot have both $0$ and $r^2$ as subsequential limits.
Suppose for contradiction that both occur. Since $h$ is continuous and strictly
positive on $(0,r^2)$, we can choose
$
0<a<a'<b'<b<r^2
$
such that, with $I=[a,b]$ and $I'=[a',b']$,
$$
\inf_{v\in I}h(v)>0.
$$
By Lemma~\ref{lem:small_jumps}, on $\{\sigma_{m,\eps}=\infty\}$, we have that a.s.,
$
V_{n+1}-V_n\to0
$
Hence, almost surely on this event, there exists a finite random time $N$ such that
$$
|V_{n+1}-V_n|<\frac{b'-a'}{2}
\quad\text{for all }n\ge N.
$$
Since both $0$ and $r^2$ are subsequential limits, there are arbitrarily large
times at which $V_n<a$ and arbitrarily large times at which $V_n>b$. Therefore,
after time $N$, the process must make infinitely many passages from below $a$
to above $b$, or from above $b$ to below $a$. During each such passage, the process must visit $I'=[a',b']$. Indeed, if the
process goes from below $a$ to above $b$ without visiting $I'$, then at some step
it must jump from below $a'$ to above $b'$, which has size at least $b'-a'$,
contradicting the bound on the jumps. The passage from above $b$ to below $a$ is
handled similarly. Hence $V_n\in I'$ infinitely often. This contradicts Lemma~\ref{lem:one_sided_exit}, which says that, on
$\{\sigma_{m,\eps}=\infty\}$, the process $(V_n)$ visits $I'$ only finitely
many times. Therefore $0$ and $r^2$ cannot both be subsequential limits.

Consequently, on $\{\sigma_{m,\eps}=\infty\}$, the bounded sequence $(V_n)$ has
exactly one subsequential limit. Hence $(V_n)$ converges, and its limit belongs
to $\{0,r^2\}$. That is,
$$
V_n=\|G_n\|^2 \to V_\infty
\quad\text{with}\quad
V_\infty\in\{0,r^2\}
$$
almost surely on $\{\sigma_{m,\eps}=\infty\}$.

Since Lemma~\ref{lem:S.bound} implies
$
\mathbb P\left(\bigcup_{m\ge1}\{\sigma_{m,\eps}=\infty\}\right)=1,
$
we obtain that a.s. $
\|G_n\|^2\to V_\infty\in\{0,r^2\}$.
Finally, 
since $G_n=n^{-\gamma}S_n$, this proves
$\operatorname{dist}\bigl(n^{-\gamma}S_n,\{0\}\cup \mathcal C_r\bigr)\to0$ a.s.
\end{proof}

\begin{proof}[Proof of Theorem~\ref{thm.nonlinear}(c)]
By Proposition~\ref{lem:dist} and its proof, $\|G_n\|$ converges almost surely to either $0$ or $r$. Let $\Gamma_0:=\{G_n\to0\}$. We apply Proposition~\ref{lem:discrete_antitrap_from_CT} (see in the appendix) on $\Gamma_0$ to the recursion~\eqref{G.SAeq}, which can be rewritten in the form
$$
G_{n+1}-G_n=\upsilon_{n+1}F_n+c_{n+1}(\varepsilon_{n+1}+r_{n+1}),
$$
where
$$
\begin{gathered}
\upsilon_{n}:=\frac1{n},\quad c_n:=n^{-\gamma},\quad F_n:=F(G_n),\quad
\varepsilon_{n}:=U_{n},\quad
r_{n+1}:=(n+1)^{\gamma-1}R_n.
\end{gathered}
$$
Since $\gamma>1/2$, we have $\sum_{n=1}^{\infty}c_n^2<\infty$ and
$$
\sup_{k\ge n+1}kc_k^2=(n+1)^{1-2\gamma}
\le2^{2\gamma}\sum_{k=n+1}^{2n}c_k^2
\quad \text{for each }n\ge1.
$$
We have verified~\eqref{eq:antitrap_coefficients}.
By conditional Jensen's inequality and Assumption~\ref{assump},
$$
\sup_{n\ge0}\E\|U_{n+1}\|^p
\le2^p\sup_{n\ge0}\E\|X_{n+1}\|^p<\infty.
$$
The noise $(U_{n+1})$ is centered conditionally on $\mathcal F_n$. Since
$\langle x,F(x)\rangle=\|x\|^{\alpha+1}-\gamma\|x\|^2,$
we can choose $\rho>0$ so that this scalar product is nonnegative for $\|x\|<2\rho$. On $\Gamma_0$, by the continuity of $F$ and the fact that $F(0)=0$, we have $\upsilon_{n+1}\|F(G_n)\|\to0$. Moreover, since $\gamma\le1$,
$$
\E[\|U_{n+1}\|^2\mid\mathcal F_n]
=\tr(\Sigma)-n^{2(\gamma-1)}\|G_n\|^{2\alpha}
\rightarrow\tr(\Sigma)>0\quad\text{on }\Gamma_0.
$$
Finally, by~\eqref{inq.Rn} and the boundedness of $G_n$ on $\Gamma_0$, we notice that
$\|r_{n+1}\|\le Cn^{\gamma-2}$
eventually on this event. Thus $\sum_{n\ge1}\|r_n\|^2<\infty$ on $\Gamma_0$. All the hypotheses of Proposition~\ref{lem:discrete_antitrap_from_CT} hold, and we have $\P(\Gamma_0)=0$. We conclude that $\|G_n\|\to r$ almost surely. We have proved the desired result.
\end{proof}

\section{Proofs of Theorems \ref{thm.1} and \ref{thm.2}}\label{sec:thm1}

We assume that Assumptions~\ref{assump} and~\ref{assump2} are fulfilled throughout this section. 
Recall that $A$ is a $d\times d$ matrix such that $0<\lambda_{\min}(A)\le \lambda_{\max}(A)<1$. For $n\ge1$, define the product of matrices
$$
\Phi(n):=\prod_{j=1}^{n-1}\left(I_d+\frac Aj\right).
$$
We first prove the following result:
\begin{lemma}\label{lem:product}
There exists an invertible matrix $\Phi_\infty$, commuting with $A$, such that
\begin{align}\label{eq:linear-product-rate}
n^{-A}\Phi(n)=\Phi_\infty+O(n^{-1}).
\end{align}
For $a>\lambda_{\max}(A)$ and $b<\lambda_{\min}(A)$, we have
\begin{align}\label{Phi.inq}
\|\Phi(n)\Phi(k)^{-1}\|
\le
\begin{cases}
C(n/k)^a&\text{for }1\le k\le n,\\
C(k/n)^{-b}&\text{for } k\ge n\ge1.
\end{cases}
\end{align}
Moreover, for $n,k\ge1$, we have
\begin{align}\label{eq:linear-product-error}
\|k^An^{-A}\Phi(n)\Phi(k)^{-1}-I_d\|
\le C(n^{-1}+k^{-1}).
\end{align}
\end{lemma}
\begin{proof}
Since $\lambda_{\min}(A)>0$, we note that $I_d+A/j$ is invertible for every $j\ge1$. Choose an integer $j_0>2\|A\|$. For $j\ge j_0$, we have
$$
\log(I_d+A/j)-A/j
=\sum_{\ell=2}^\infty\frac{(-1)^{\ell+1}A^\ell}{\ell j^\ell},
\quad
\|\log(I_d+A/j)-A/j\|\le2\|A\|^2j^{-2}.
$$
Since the factors in $\Phi(n)$ commute, we obtain, for $n\ge j_0$,
\begin{align}\label{eq:linear-product-log}
n^{-A}\Phi(n)
=\Phi(j_0)\exp\left\{
A\left(\sum_{j=j_0}^{n-1}\frac1j-\log n\right)
+\sum_{j=j_0}^{n-1}\left(\log(I_d+A/j)-A/j\right)
\right\}.
\end{align}
Since
$0\le1/j-\log(1+1/j)\le j^{-2}$ for all $j\ge1$,
we infer that both terms in the exponent in~\eqref{eq:linear-product-log} converge with errors $O(n^{-1})$. We therefore obtain~\eqref{eq:linear-product-rate}, with an invertible limit $\Phi_\infty$ commuting with $A$. Taking inverses in~\eqref{eq:linear-product-rate} and including the finitely many indices $n<j_0$, we also have
\begin{align}\label{normbnd}
\sup_{n\ge1}\left(\|n^{-A}\Phi(n)\|
+\|(n^{-A}\Phi(n))^{-1}\|\right)<\infty.
\end{align}

For a Jordan block $\lambda I_m+N_m$, we notice that
\begin{align}\label{eq:linear-jordan-power}
t^{\pm(\lambda I_m+N_m)}
=t^{\pm\lambda}\sum_{j=0}^{m-1}\frac{(\pm\log t)^j}{j!}N_m^j
\quad  \text{for }t\ge1.
\end{align}
By~\eqref{eq:linear-jordan-power}, for $a>\lambda_{\max}(A)$ and $b<\lambda_{\min}(A)$, we have $\|t^A\|\le Ct^a$ and $\|t^{-A}\|\le Ct^{-b}$ for $t\ge1$. Combining these bounds with the fact that
$\Phi(n)\Phi(k)^{-1}=(n/k)^A(n^{-A}\Phi(n))(k^{-A}\Phi(k))^{-1}$ and \eqref{normbnd}, we obtain~\eqref{Phi.inq}. Finally, using~\eqref{eq:linear-product-rate} and \eqref{normbnd}, we have
\begin{align*}
\|k^An^{-A}\Phi(n)\Phi(k)^{-1}-I_d\|
&\le\|n^{-A}\Phi(n)-k^{-A}\Phi(k)\|
\,\|(k^{-A}\Phi(k))^{-1}\|\le C(n^{-1}+k^{-1}).
\end{align*}
Thus we obtain~\eqref{eq:linear-product-error}.
\end{proof}

We define $U_1:=X_1$ and $U_{n+1}:=X_{n+1}-AS_n/n$ for $n\ge1$. By Assumptions~\ref{assump} and~\ref{assump2}, we have $\E[U_k\mid\mathcal F_{k-1}]=0$ for $k\ge1$. Using conditional Jensen's inequality, we also have
\begin{align}\label{U.inq}
\sup_{k\ge1}\E\|U_k\|^p
\le2^p\sup_{k\ge1}\E\|X_k\|^p<\infty.
\end{align}
From the definition of $U_{n+1}$, we have
\begin{align}\label{S.ind}
S_{n+1}=(I_d+A/n)S_n+U_{n+1}\quad\text{for } n\ge1.
\end{align}
Iterating~\eqref{S.ind}, we obtain
\begin{align}\label{M.ind}
M_n:=\Phi(n)^{-1}S_n
=\sum_{k=1}^n\Phi(k)^{-1}U_k\quad\text{for } n\ge1.
\end{align}
Setting $M_0:=0$, we infer from~\eqref{U.inq} and~\eqref{M.ind} that $(M_n)_{n\ge0}$ is a square-integrable martingale. We choose $a\in(1/2,1)$ with $a>\lambda_{\max}(A)$. Using martingale orthogonality and~\eqref{Phi.inq}, we have
\begin{align}\label{S.lim}
\E\|S_n\|^2
=\sum_{k=1}^n\E\|\Phi(n)\Phi(k)^{-1}U_k\|^2
\le Cn^{2a}\sum_{k=1}^nk^{-2a}\le C_1n^{2a}.
\end{align}
Since
$\E[U_{n+1}U_{n+1}^\top\mid\mathcal F_n]
=\Sigma+\mathcal E_n-n^{-2}AS_nS_n^\top A^\top$ for each $n\ge1$,  we conclude from~\eqref{S.lim} and Assumption~\ref{assump2} that
\begin{align}\label{eq:linear-noise-covariance}
\E\big[\left\|\E[U_kU_k^\top\mid\mathcal F_{k-1}]-\Sigma\right\|\big]
\rightarrow0\quad\text{as }k\to\infty.
\end{align}

\begin{lemma}\label{lem:weighted-array}
Let $(U_k)_{k\ge1}$ be a martingale difference sequence with respect to $(\F_k)_{k\ge 1}$ satisfying~\eqref{U.inq} and~\eqref{eq:linear-noise-covariance}. Let $B_{n,k}\in\R^{d\times d}$, $n\ge3$, $k\ge1$, be deterministic matrices such that $\sum_{k=1}^\infty\|B_{n,k}\|^2<\infty$ for every $n\ge3$. Suppose that
\begin{align}\label{eq:linear-array-conditions}
\sup_{k\ge1}\|B_{n,k}\|\rightarrow0,\quad
\sum_{k=1}^\infty B_{n,k}\Sigma B_{n,k}^\top\rightarrow\Sigma_0.
\end{align}
Then, for each fixed $n\ge3$, the series
$\sum_{k=1}^{\infty}B_{n,k}U_k$ converges in $L^2$, and, as $n\to\infty$,
$$
\sum_{k=1}^{\infty}B_{n,k}U_k
\convd\mathcal N_d(0,\Sigma_0).
$$
\end{lemma}
\begin{proof}
Since $\Sigma$ is positive definite, we have
\begin{align}\label{eq:linear-array-bound}
\sum_{k=1}^\infty\|B_{n,k}\|^2
\le\lambda_{\min}(\Sigma)^{-1}
\tr\left(\sum_{k=1}^\infty B_{n,k}\Sigma B_{n,k}^\top\right)\le C.
\end{align}
Using~\eqref{U.inq},~\eqref{eq:linear-array-bound} and martingale orthogonality, we conclude that, for each fixed $n\ge3$, the series $\sum_{k=1}^{\infty}B_{n,k}U_k$ converges in $L^2$. For $K\ge1$, we have
\begin{align}\label{eq:linear-array-covariance-error}
\E\left\|\sum_{k=1}^\infty B_{n,k}
\left(\E[U_kU_k^\top\mid\mathcal F_{k-1}]-\Sigma\right)B_{n,k}^\top\right\|&\le\left(\sup_{k\ge1}\|B_{n,k}\|^2\right)
\sum_{k=1}^K\E\left\|\E[U_kU_k^\top\mid\mathcal F_{k-1}]-\Sigma\right\|\notag\\
&\quad+C\sup_{k>K}\E\left\|\E[U_kU_k^\top\mid\mathcal F_{k-1}]-\Sigma\right\|.
\end{align}
Letting first $n\to\infty$ and then $K\to\infty$ in~\eqref{eq:linear-array-covariance-error}, we obtain convergence of the conditional covariance sum to $\Sigma_0$ in $L^1$. For every $\eps>0$, using~\eqref{U.inq}, we also have
\begin{align}\label{eq:linear-array-lindeberg}
\sum_{k=1}^\infty\E\left[\|B_{n,k}U_k\|^2
\mathbf1_{\{\|B_{n,k}U_k\|>\eps\}}\right]
\le C\eps^{2-p}\left(\sup_{k\ge1}\|B_{n,k}\|\right)^{p-2}
\sum_{k=1}^\infty\|B_{n,k}\|^2\rightarrow0.
\end{align}
We choose integers $k_n\ge n$ such that $\sum_{k=k_n+1}^\infty\|B_{n,k}\|^2\le n^{-2}$. By~\eqref{U.inq}, we have
\begin{align}\label{eq:linear-array-truncation}
\E\tr\left(\sum_{k=k_n+1}^\infty B_{n,k}
\E[U_kU_k^\top\mid\mathcal F_{k-1}]B_{n,k}^\top\right)
=\E\left\|\sum_{k=k_n+1}^\infty B_{n,k}U_k\right\|^2\le Cn^{-2}.
\end{align}
By~\eqref{eq:linear-array-covariance-error},~\eqref{eq:linear-array-truncation} and Markov's inequality applied to~\eqref{eq:linear-array-lindeberg}, we verify the hypotheses of Proposition~\ref{thm.mtg} (see in the appendix) for $B_{n,k}U_k$, with $1\le k\le k_n$. We therefore have
$
\sum_{k=1}^{k_n}B_{n,k}U_k\convd\mathcal N_d(0,\Sigma_0).
$
By~\eqref{eq:linear-array-truncation}, we also have
$
\sum_{k=k_n+1}^{\infty}B_{n,k}U_k\rightarrow0\quad\text{in }L^2.
$
Hence,
$
\sum_{k=1}^{\infty}B_{n,k}U_k \convd\mathcal N_d(0,\Sigma_0).
$
\end{proof}

For later use, we record the following consequence of~\eqref{eq:linear-product-error}. For fixed $0<u<v<\infty$, we have
$$
\sup_{\substack{k\in\N\\u\le k/n\le v}}
\|\Phi(n)\Phi(k)^{-1}-(k/n)^{-A}\|\le C/n.
$$
Since $t\mapsto t^{-A}\Sigma t^{-A^\top}$ has bounded derivative on $[u,v]$, comparing the sum with the integral over intervals of length $1/n$ and bounding the two boundary intervals by $C/n$, we obtain
\begin{align}\label{eq:linear-compact-covariance}
\frac1n\sum_{k=\lfloor un\rfloor+1}^{\lfloor vn\rfloor}
\Phi(n)\Phi(k)^{-1}\Sigma\bigl(\Phi(n)\Phi(k)^{-1}\bigr)^\top
\rightarrow\int_u^v t^{-A}\Sigma t^{-A^\top}\rmd t.
\end{align}

\subsection{Subcritical case}
\begin{proof}[Proof of Theorem~\ref{thm.1}(a)]
We choose $\lambda_{\max}(A)<a<1/2$ and define
$$
B_{n,k}:=n^{-1/2}\Phi(n)\Phi(k)^{-1}\mathbf1_{\{1\le k\le n\}}
\quad\text{for } n\ge3,\, k\ge1.
$$
Using~\eqref{M.ind} and~\eqref{Phi.inq}, we have
\begin{align}\label{eq:linear-subcritical-coefficients}
\frac{S_n}{\sqrt n}=\sum_{k=1}^n B_{n,k}U_k,\quad
\sup_{k\ge1}\|B_{n,k}\|\le Cn^{a-1/2}\rightarrow0,\quad
\sum_{k=1}^n\|B_{n,k}\|^2\le C.
\end{align}
For $0<\delta<1$, by~\eqref{Phi.inq} and~\eqref{eq:linear-jordan-power}, we also have
\begin{align}\label{eq:linear-subcritical-endpoint}
\frac1n\sum_{k=1}^{\lfloor\delta n\rfloor}\|\Phi(n)\Phi(k)^{-1}\|^2
&\le Cn^{2a-1}\sum_{k=1}^{\lfloor\delta n\rfloor}k^{-2a}
\le C_1\delta^{1-2a},\notag\\
\int_0^\delta\|t^{-A}\Sigma t^{-A^\top}\|\rmd t
&\le C\int_0^\delta t^{-2a}\rmd t
=\frac{C}{1-2a}\delta^{1-2a}.
\end{align}
We choose the constants in~\eqref{eq:linear-subcritical-endpoint} independently of $n,\delta$. Combining~\eqref{eq:linear-compact-covariance} on $[\delta,1]$ with~\eqref{eq:linear-subcritical-endpoint}, and letting first $n\to\infty$ and then $\delta\downarrow0$, we obtain
\begin{align}\label{eq:linear-subcritical-covariance}
\sum_{k=1}^n B_{n,k}\Sigma B_{n,k}^\top
\rightarrow\Xi:=\int_0^1t^{-A}\Sigma t^{-A^\top}\rmd t.
\end{align}
Set $B:=\tfrac12I_d-A$.  Substituting $t=e^{-u}$ in~\eqref{eq:linear-subcritical-covariance} and integrating by parts, we have
\begin{align}\label{eq:linear-lyapunov-identity}
B\Xi+\Xi B^\top
=-\int_0^\infty\frac{\rmd}{\rmd u}
\left(e^{-Bu}\Sigma e^{-B^\top u}\right)\rmd u
=\Sigma.
\end{align}
If $\Xi_1,\Xi_2$ both solve~\eqref{eq.Lypunov}, differentiating
$e^{-Bu}(\Xi_1-\Xi_2)e^{-B^\top u}$ with respect to $u$, we obtain
$$
\frac{\mathrm d}{\mathrm d u}
\left(e^{-Bu}(\Xi_1-\Xi_2)e^{-B^\top u}\right)
=-e^{-Bu}\left[B(\Xi_1-\Xi_2)+(\Xi_1-\Xi_2)B^\top\right]
e^{-B^\top u}=0.
$$
Thus, for every $u\ge0$, we have
\begin{align}\label{eq:linear-lyapunov-uniqueness}
\Xi_1-\Xi_2=e^{-Bu}(\Xi_1-\Xi_2)e^{-B^\top u}\rightarrow0
\quad\text{as }u\to\infty.
\end{align}
Here we used the fact that $\lambda_{\min}(B)>0$. Therefore \eqref{eq.Lypunov} has a unique solution. By~\eqref{eq:linear-subcritical-coefficients},~\eqref{eq:linear-subcritical-covariance} and Lemma~\ref{lem:weighted-array}, we conclude the proof.
\end{proof}

\subsection{Supercritical case}
\begin{proof}[Proof of Theorem~\ref{thm.1}(b)]
We choose $1/2<b<\lambda_{\min}(A)$. Using~\eqref{M.ind},~\eqref{U.inq} and~\eqref{Phi.inq}, we have
\begin{align}\label{M.bound}
\sup_{n\ge1}\E\|M_n\|^2
\le C\sum_{k=1}^\infty\|\Phi(k)^{-1}\|^2
\le C_1\sum_{k=1}^\infty k^{-2b}<\infty.
\end{align}
By the $L^2$ martingale convergence theorem, we have
\begin{align}\label{eq:linear-supercritical-limit}
n^{-A}S_n\rightarrow W:=\Phi_\infty\sum_{k=1}^\infty\Phi(k)^{-1}U_k
\quad\text{almost surely and in }L^2,
\end{align}
where we used~\eqref{eq:linear-product-rate} for the convergence of $n^{-A}S_n$. We also have $\E W=0$. By~\eqref{eq:linear-noise-covariance}, we can choose $k_0\ge1$ such that
$$
v^\top\E[U_kU_k^\top]v\ge\tfrac12\lambda_{\min}(\Sigma)\|v\|^2
\quad\text{for } v\in\R^d,\, k\ge k_0.
$$
Using martingale orthogonality and invertibility of $\Phi_\infty$, we have, for $v\in\R^d\setminus\{0\}$,
$$
v^\top\operatorname{Cov}(W)v
=\sum_{k=1}^\infty\E\left|v^\top\Phi_\infty\Phi(k)^{-1}U_k\right|^2
\ge\tfrac12\lambda_{\min}(\Sigma)
\|\bigl(\Phi(k_0)^{-1}\bigr)^\top\Phi_\infty^\top v\|^2>0.
$$
We have therefore proved non-degeneracy of $W$.

For the fluctuations, we define
$$
B_{n,k}:=-n^{-1/2}\Phi(n)\Phi(k)^{-1}\mathbf1_{\{k\ge n+1\}}\quad\text{for } n\ge3,\, k\ge1.
$$
By~\eqref{Phi.inq}, we have
\begin{align}\label{eq:linear-supercritical-coefficients}
\sup_{k\ge1}\|B_{n,k}\|\le Cn^{-1/2},
\quad \sum_{k=n+1}^\infty\|B_{n,k}\|^2\le C.
\end{align}
For $L>1$, using~\eqref{Phi.inq} and~\eqref{eq:linear-jordan-power}, we also have
\begin{align}\label{eq:linear-supercritical-endpoint}
\frac1n\sum_{k=\lfloor Ln\rfloor+1}^\infty\|\Phi(n)\Phi(k)^{-1}\|^2
&\le Cn^{2b-1}\sum_{k=\lfloor Ln\rfloor+1}^\infty k^{-2b}
\le C_1 L^{1-2b},\notag\\
\int_L^\infty\|t^{-A}\Sigma t^{-A^\top}\|\rmd t
&\le\frac{C}{2b-1}L^{1-2b}.
\end{align}
We choose the constants in~\eqref{eq:linear-supercritical-endpoint} independently of $n,L$. Combining~\eqref{eq:linear-compact-covariance} on $[1,L]$ with~\eqref{eq:linear-supercritical-endpoint}, and letting first $n\to\infty$ and then $L\to\infty$, we have that
\begin{align}\label{eq:linear-supercritical-covariance}
\sum_{k=n+1}^\infty B_{n,k}\Sigma B_{n,k}^\top
\rightarrow-\Xi:=\int_1^\infty t^{-A}\Sigma t^{-A^\top}\rmd t.
\end{align}
Substituting $t=e^u$ in~\eqref{eq:linear-supercritical-covariance} and applying~\eqref{eq:linear-lyapunov-identity} and~\eqref{eq:linear-lyapunov-uniqueness} with $A$ replaced by $I_d-A$, we infer that $\Xi$ is the unique solution of~\eqref{eq.Lypunov}.

By~\eqref{eq:linear-supercritical-coefficients},~\eqref{eq:linear-supercritical-covariance} and Lemma~\ref{lem:weighted-array}, we have $\sum_{k=n+1}^\infty B_{n,k}U_k\convd\mathcal N_d(0,-\Xi)$. Using~\eqref{M.ind} and~\eqref{eq:linear-supercritical-limit}, we also have
\begin{align}\label{eq:linear-centering}
\frac{S_n-n^AW}{\sqrt n}-\sum_{k=n+1}^\infty B_{n,k}U_k
&=\frac{(\Phi(n)\Phi_\infty^{-1}-n^A)W}{\sqrt n},\notag\\
\E\left\|\frac{(\Phi(n)\Phi_\infty^{-1}-n^A)W}{\sqrt n}\right\|^2
&\le Cn^{2a-3}\E\|W\|^2\rightarrow0,
\quad \lambda_{\max}(A)<a<1.
\end{align}
For the second line of~\eqref{eq:linear-centering}, we applied~\eqref{eq:linear-product-rate} and~\eqref{eq:linear-jordan-power}. Hence $n^{-1/2}(S_n-n^AW)\convd\mathcal N_d(0,-\Xi)$.
\end{proof}

\subsection{Critical case and joint convergence}

\begin{proof}[Proofs of Theorems~\ref{thm.2} and~\ref{thm.1}(c)]
We recall that $P_s,P_c,P_u$ are the projections corresponding to the decomposition~\eqref{eq:linear-spectral-decomposition}. For the estimates~\eqref{eq:linear-projected-coefficients} and~\eqref{eq:linear-mixed-bounds}, we choose $\lambda_{\max}(A_s)<a_s<1/2$ when $E_s\ne\{0\}$.

Applying~\eqref{M.bound} on $E_u$ and using~\eqref{eq:linear-product-rate}, we have
\begin{align}\label{eq:linear-unstable-projection-limit}
n^{-A_u}P_uS_n\rightarrow W_u:=\sum_{k=1}^\infty P_u\Phi_\infty\Phi(k)^{-1}U_k
\quad\text{almost surely and in }L^2,
\end{align}
where the series converges almost surely and in $L^2$ by the martingale convergence theorem.

For $n\ge3$ and $k\ge1$, we define
\begin{align}\label{eq:linear-joint-coefficients}
B_{n,k}=
\begin{pmatrix}B^s_{n,k}\\B^c_{n,k}\\B^u_{n,k}\end{pmatrix}
:=\begin{pmatrix}
n^{-1/2}P_s\Phi(n)\Phi(k)^{-1}\mathbf1_{\{1\le k\le n\}}\\
D_n^{-1}k^{-A_c}P_c\mathbf1_{\{1\le k\le n\}}\\
-n^{-1/2}P_u\Phi(n)\Phi(k)^{-1}\mathbf1_{\{k\ge n+1\}}
\end{pmatrix}.
\end{align}
Applying~\eqref{eq:linear-subcritical-coefficients} and~\eqref{eq:linear-supercritical-coefficients} on $E_s$ and $E_u$, respectively, we have
\begin{align}\label{eq:linear-projected-coefficients}
\sup_{k\ge1}\|B^s_{n,k}\|\le Cn^{a_s-1/2},\quad
\sup_{k\ge1}\|B^u_{n,k}\|\le Cn^{-1/2},\quad
\sum_{k=1}^\infty(\|B^s_{n,k}\|^2+\|B^u_{n,k}\|^2)\le C.
\end{align}
Applying the covariance estimates~\eqref{eq:linear-subcritical-covariance} and~\eqref{eq:linear-supercritical-covariance} on the corresponding subspaces, with $\Sigma$ replaced by $P_s\Sigma P_s^\top$ and $P_u\Sigma P_u^\top$, respectively, we obtain
\begin{align}\label{eq:linear-projected-covariances}
\sum_{k=1}^nB^s_{n,k}\Sigma(B^s_{n,k})^\top
&\rightarrow\int_0^1t^{-A_s}P_s\Sigma P_s^\top t^{-A_s^\top}\rmd t=\Xi_s,\notag\\
\sum_{k=n+1}^\infty B^u_{n,k}\Sigma(B^u_{n,k})^\top
&\rightarrow\int_1^\infty t^{-A_u}P_u\Sigma P_u^\top t^{-A_u^\top}\rmd t=\Xi_u.
\end{align}
By~\eqref{eq:linear-lyapunov-identity} and~\eqref{eq:linear-lyapunov-uniqueness}, we identify $\Xi_s,\Xi_u$ with the Lyapunov solutions in Theorem~\ref{thm.2}.

We next consider $E_c\ne\{0\}$, retaining $Q,D_n$ and $\widetilde\Sigma$ from Theorem~\ref{thm.2}. For $1\le r\le R$, $1\le i\le j\le m_r$ and $1\le k\le n$, using~\eqref{eq:linear-jordan-power}, we have
\begin{align}\label{eq:linear-critical-entry}
\bigl(Q^{-1}D_n^{-1}k^{-A_c}Q\bigr)_{i,j}(r,r)
=\frac{(-1)^{j-i}k^{-\lambda_r}(\log k)^{j-i}}
{(j-i)!(\log n)^{m_r-i+1/2}}.
\end{align}
We note that all other entries of the matrix in~\eqref{eq:linear-critical-entry} are zero. Here we take $(\log1)^0=1$. Since $\Re\lambda_r=1/2$, $k\le n$ and $j\le m_r$, we deduce from~\eqref{eq:linear-critical-entry} that
\begin{align}\label{eq:linear-critical-bound}
\|D_n^{-1}k^{-A_c}\|\le\frac{C}{\sqrt{k\log n}}
\quad \text{for each } 1\le k\le n,\, n\ge3.
\end{align}
By~\eqref{eq:linear-joint-coefficients} and~\eqref{eq:linear-critical-bound}, we have
\begin{align}\label{eq:linear-critical-energy}
\sup_{k\ge1}\|B^c_{n,k}\|\le\frac{C}{\sqrt{\log n}},\quad
\sum_{k=1}^n\|B^c_{n,k}\|^2
\le\frac{C}{\log n}\sum_{k=1}^n\frac1k\le C_1.
\end{align}

To compute the critical covariance, we first establish, for each $q\in\Z_+$ and $\omega\in\R$, that
\begin{align}\label{eq:linear-log-sums}
\sum_{k=1}^nk^{-1-\mathrm i\omega}(\log k)^q
=\begin{cases}
\displaystyle\frac{(\log n)^{q+1}}{q+1}+O(1),&\omega=0,\\[2mm]
O(1+(\log n)^q),&\omega\ne0.
\end{cases}
\end{align}
Indeed, we have
\begin{align*}
&\left|\sum_{k=1}^nk^{-1-\mathrm i\omega}(\log k)^q
-\int_1^nt^{-1-\mathrm i\omega}(\log t)^q\rmd t\right|\le n^{-1}(\log n)^q
+\int_1^\infty\left|\frac{\rmd}{\rmd t}
\left(t^{-1-\mathrm i\omega}(\log t)^q\right)\right|\rmd t
\le C,
\end{align*}
where $C$ may depend on $q,\omega$. Substituting $t=e^u$, we obtain the first case of~\eqref{eq:linear-log-sums}. For $\omega\ne0$, we have $|\int_0^{\log n}e^{-\mathrm i\omega u}\rmd u|\le2/|\omega|$ and, for $q\ge1$,
$$
\int_0^{\log n}e^{-\mathrm i\omega u}u^q\rmd u
=-\frac{e^{-\mathrm i\omega\log n}(\log n)^q}{\mathrm i\omega}
+\frac{q}{\mathrm i\omega}\int_0^{\log n}e^{-\mathrm i\omega u}u^{q-1}\rmd u.
$$
By induction on $q$, we obtain the second case of~\eqref{eq:linear-log-sums}.

For $1\le r,s\le R$, $1\le i\le m_r$ and $1\le j\le m_s$, expanding the matrix products with~\eqref{eq:linear-critical-entry}, we have
\begin{align}\label{eq:linear-critical-covariance-entry}
&\left(Q^{-1}\sum_{k=1}^nB^c_{n,k}\Sigma(B^c_{n,k})^\top(Q^{-1})^*\right)_{i,j}(r,s)\notag\\
&\quad=\sum_{a=i}^{m_r}\sum_{b=j}^{m_s}
\frac{(-1)^{a+b-i-j}\widetilde\Sigma_{a,b}(r,s)}
{(a-i)!(b-j)!(\log n)^{m_r+m_s-i-j+1}}
\sum_{k=1}^nk^{-1-\mathrm i\Im(\lambda_r-\lambda_s)}(\log k)^{a+b-i-j}.
\end{align}
For $i\le a\le m_r$ and $j\le b\le m_s$, we have $(m_r+m_s-i-j)-(a+b-i-j)=(m_r-a)+(m_s-b)\ge0$. Since $\Re\lambda_r=\Re\lambda_s=1/2$, we infer from~\eqref{eq:linear-log-sums} that
\begin{align}\label{eq:linear-normalized-log-sums}
\frac{\sum_{k=1}^nk^{-1-\mathrm i\Im(\lambda_r-\lambda_s)}(\log k)^{a+b-i-j}}
{(\log n)^{m_r+m_s-i-j+1}}
\rightarrow
\begin{cases}
\displaystyle\frac1{m_r+m_s-i-j+1},&\lambda_r=\lambda_s,\ a=m_r,\ b=m_s,\\[2mm]
0,&\text{otherwise}.
\end{cases}
\end{align}
Substituting~\eqref{eq:linear-normalized-log-sums} in~\eqref{eq:linear-critical-covariance-entry}, we recover the entries of $\mathcal V$ in Theorem~\ref{thm.2}. We therefore have
\begin{align}\label{eq:linear-critical-covariance-limit}
\sum_{k=1}^nB^c_{n,k}\Sigma(B^c_{n,k})^\top\rightarrow Q\mathcal VQ^*=\Xi_c.
\end{align}

We now check the mixed covariances. By~\eqref{eq:linear-joint-coefficients}, for every $k\ge1$, we have
$$
B^s_{n,k}\Sigma(B^u_{n,k})^\top
=B^c_{n,k}\Sigma(B^u_{n,k})^\top=0.
$$
If $E_s$ and $E_c$ are nonzero, we fix $0<\delta<1$. Using~\eqref{Phi.inq} on $E_s$ and~\eqref{eq:linear-critical-bound}, we have, for all sufficiently large $n$,
\begin{align}\label{eq:linear-mixed-bounds}
\sum_{k=1}^{\lfloor\delta n\rfloor}\|B^s_{n,k}\|^2
&\le Cn^{2a_s-1}\sum_{k=1}^{\lfloor\delta n\rfloor}k^{-2a_s}
\le C_1\delta^{1-2a_s},\notag\\
\sum_{k=\lfloor\delta n\rfloor+1}^n\|B^c_{n,k}\|^2
&\le\frac C{\log n}\sum_{k=\lfloor\delta n\rfloor+1}^n\frac1k
\le\frac{C_1(1+\log(2/\delta))}{\log n}.
\end{align}
We choose $C,C_1$ in~\eqref{eq:linear-mixed-bounds} independently of $n,\delta$. Applying the Cauchy--Schwarz inequality and using~\eqref{eq:linear-projected-coefficients},~\eqref{eq:linear-critical-energy} and~\eqref{eq:linear-mixed-bounds}, we obtain
\begin{align}\label{eq:linear-mixed-covariance}
\left\|\sum_{k=1}^nB^s_{n,k}\Sigma(B^c_{n,k})^\top\right\|
&\le C\left(\sum_{k=1}^{\lfloor\delta n\rfloor}\|B^s_{n,k}\|^2\right)^{1/2}
+C\left(\sum_{k=\lfloor\delta n\rfloor+1}^n\|B^c_{n,k}\|^2\right)^{1/2}\notag\\
&\le C_1\delta^{(1-2a_s)/2}
+C_2\left(\frac{1+\log(2/\delta)}{\log n}\right)^{1/2}.
\end{align}
Letting first $n\to\infty$ and then $\delta\downarrow0$ in~\eqref{eq:linear-mixed-covariance}, we have that
$$
\sum_{k=1}^n B^s_{n,k}\Sigma(B^c_{n,k})^\top
\rightarrow0.
$$
Combining this with~\eqref{eq:linear-projected-coefficients},~\eqref{eq:linear-projected-covariances},~\eqref{eq:linear-critical-energy} and~\eqref{eq:linear-critical-covariance-limit}, we verify~\eqref{eq:linear-array-conditions}. By Lemma~\ref{lem:weighted-array}, we infer that
\begin{align}\label{eq:linear-joint-array-limit}
\sum_{k=1}^\infty B_{n,k}U_k
\convd\mathcal N_d\bigl(0,\diag(\Xi_s,\Xi_c,\Xi_u)\bigr).
\end{align}

It remains to identify the components in~\eqref{eq:linear-joint-array-limit}. By~\eqref{M.ind} and~\eqref{eq:linear-joint-coefficients}, we have
$$
\sum_{k=1}^nB^s_{n,k}U_k=\frac{P_sS_n}{\sqrt n}.
$$
Using~\eqref{eq:linear-product-error} on $E_c$,~\eqref{M.ind},~\eqref{eq:linear-critical-bound} and martingale orthogonality, we obtain
\begin{align}\label{eq:linear-critical-approximation}
\E\left\|D_n^{-1}n^{-A_c}P_cS_n-\sum_{k=1}^nB^c_{n,k}U_k\right\|^2
&\le C\sum_{k=1}^n k^{-2}\|D_n^{-1}k^{-A_c}\|^2\le\frac C{\log n}\sum_{k=1}^nk^{-3}\rightarrow0.
\end{align}
For $\lambda_{\max}(A_u)<a_u<1$, applying~\eqref{eq:linear-centering} on $E_u$ with the limit in~\eqref{eq:linear-unstable-projection-limit}, we also have
\begin{align}\label{eq:linear-projected-centering}
\E\left\|\frac{P_uS_n-n^{A_u}W_u}{\sqrt n}
-\sum_{k=n+1}^\infty B^u_{n,k}U_k\right\|^2
\le Cn^{2a_u-3}\E\|W_u\|^2\rightarrow0.
\end{align}
Combining~\eqref{eq:linear-critical-approximation},~\eqref{eq:linear-projected-centering} and~\eqref{eq:linear-joint-array-limit}, we conclude the proof of Theorem~\ref{thm.2}.

Finally, under the assumptions of Theorem~\ref{thm.1}(c), we have $\lambda=\tr(A)/d=1/2$. Taking $E_c=\R^d$, $R=1$ and $m_1=d$ in~\eqref{eq:linear-critical-covariance-entry}--\eqref{eq:linear-critical-covariance-limit}, and using $(-1)^{2d-i-j}=(-1)^{i+j}$, we obtain the stated covariance and conclude the proof of Theorem~\ref{thm.1}(c).
\end{proof}

\begin{appendix}

\section{Non-convergence theorems}\label{appendix:noncov}
We will apply the following non-convergence result for continuous-time processes to obtain a discrete-time version which is used in the proof of Theorem \ref{thm.nonlinear}(c).

Consider a càdlàg process $Z$ in $\mathbb{R}^d$. We define $Z_{t-} := \lim_{s \uparrow t} Z_s$ as the \textit{left limit} of $Z$ at $t>0$, and let $\Delta Z_t := Z_t - Z_{t-}$ be the \textit{jump} of $Z$ at $t>0$. For a càdlàg $\mathbb R^d$-valued process $Z$, the \textit{total variation} on an interval $(s,t]$ is denoted $V(Z,(s,t])$ and defined by
$$
V(Z,(s,t]) = \sup_{\mathcal{P}} \sum_{i=1}^n \|Z_{t_i} - Z_{t_{i-1}}\|,
$$
where the supremum is taken over all partitions $\mathcal{P} = \{s = t_0 < t_1 < \cdots < t_n = t\}$ of $(s,t]$. If $V(Z,(0,t]) < \infty$ for every $t$, we say that $Z$ has \textit{finite variation}.

The \textit{quadratic variation} of a càdlàg semimartingale $Z$ on an interval $(s,t]$ is denoted $[Z]_{s,t}$ and defined as the limit in probability of $\sum_{i=1}^n \|Z_{t_i} - Z_{t_{i-1}}\|^2$ as the mesh of the partition $\mathcal{P} = \{s = t_0 < t_1 < \cdots < t_n = t\}$ of $(s,t]$ tends to $0$. The quadratic variation of $Z$ on $(0,t]$ is written $[Z]_t$, and whenever $s \le t$ we have $[Z]_{s,t} = [Z]_t - [Z]_s$.
If $M$ is a locally square integrable local martingale, then its \textit{predictable quadratic variation} process is well defined. It is the predictable process, denoted $\langle M \rangle$, for which $[M] - \langle M \rangle$ is a local martingale.

\begin{proposition}[Raimond-Tarres 2023, \cite{RT2023}, Theorem 2.1.1]\label{thm:RT_specialized}
Let $(\mathcal F_t)_{t\ge0}$ be a complete right-continuous filtration, set $\mathcal F_\infty:=\sigma(\bigcup_{t\ge0}\mathcal F_t)$, and let
$(X_t)_{t\ge0}$ be a c\`adl\`ag $\mathbb R^d$-valued process of the form
$$
X_t-X_0=\int_0^t F_s\rmd s+M_t+R_t,
$$
where $F$ is progressively measurable, $M$ is a c\`adl\`ag locally square-integrable martingale, and $R$ is a c\`adl\`ag adapted process of finite variation. Let
$\Gamma\in\mathcal F_\infty$, and let $\alpha:[0,\infty)\to(0,\infty)$ be
nonincreasing càdlàg function with $\lim_{t\to\infty}\alpha(t)=0$.

Assume that there exist a constant $\rho>0$, a finite random time $T_0$, and a nonnegative
càdlàg adapted process $E_t$ such that a.s. on $\Gamma$,

\begin{itemize}
\item[(i)]
$\langle X_{t-},F_t\rangle\ge0\quad\text{for almost every $t\ge T_0$ such that $\|X_{t-}\|<\rho$ }$;

\item[(ii)]
$
\alpha(t)^2-\alpha(t+1)^2
=
O(\langle M\rangle_{t,t+1})
$ as $t\to\infty$;

\item[(iii)]
$
\langle M\rangle_{t,\infty}=O(\alpha(t)^2)
$ as $t\to\infty$; 

\item[(iv)]
$
V(R,(t,\infty))=o(\alpha(t))$ as $t\to\infty$;

\item[(v)]
for every $t\ge0$ and every stopping time $S>t$,
$$
\E[\|\Delta M_S\|^2\mathbf 1_{\{S<\infty\}}\mid\mathcal F_t]
\le
E_t\,\mathbb P(S<\infty\mid\mathcal F_t)
\quad\text{and}\quad
E_t=O(\alpha(t)^2).
$$
\end{itemize}

Then
$
\mathbb P(\Gamma\cap\{X_t\to0\})=0.
$
\end{proposition}

It is worth mentioning that other variations of above the non-convergence results have appeared in \cite[Theorem 6.14]{RN2021} and \cite[Theorem 5.5]{CNV2022} with applications to random processes with reinforcement. Using the continuous-time result above, we obtain the following discrete-time version, which is a  modification of Theorem 2.2.1 in \cite{RT2023}. 

\begin{proposition}\label{lem:discrete_antitrap_from_CT}
Let $(\mathcal F_n)_{n\ge0}$ be a filtration and let $(G_n)_{n\ge1}$ be an adapted $\mathbb R^d$-valued process satisfying
\begin{equation}\label{eq:antitrap_recursion}
G_{n+1}-G_n
=\upsilon_{n+1}F_n+c_{n+1}(\varepsilon_{n+1}+r_{n+1})\quad\text{for each }n\ge1,
\end{equation}
where $(\upsilon_n)_{n\ge1}$ and $(c_n)_{n\ge1}$ are deterministic nonnegative sequences satisfying
\begin{equation}\label{eq:antitrap_coefficients}
\begin{gathered}
0<\sum_{k=n+1}^{\infty}c_k^2<\infty\quad\text{for all }n\ge1,\quad
\sup_{k\ge n+1}kc_k^2
=O\left(\sum_{k=n+1}^{2n}c_k^2\right)
\quad\text{as }n\to\infty,
\end{gathered}
\end{equation}
and $F_n,\varepsilon_n,r_n$ are $\mathcal F_n$-measurable $\mathbb R^d$-valued random variables. Assume that
\begin{equation}\label{eq:antitrap_noise}
\E[\varepsilon_{n+1}\mid\mathcal F_n]=0\quad\text{for all }n\ge1
\quad\text{and}\quad\sup_{n\ge1}\E\|\varepsilon_n\|^p<\infty
\quad\text{for some }p>2.
\end{equation}
Let $\Gamma\in\sigma(\bigcup_{n\ge0}\mathcal F_n)$. Suppose that there exist deterministic constants $\rho,C_1,C_2>0$ such that, almost surely on $\Gamma$, we have $\sum_{k=1}^{\infty}\|r_k\|^2<\infty$ and, for all sufficiently large $n$,
\begin{gather}
\langle G_n,F_n\rangle\ge0\quad\text{whenever }\|G_n\|<2\rho,
\quad \upsilon_{n+1}\|F_n\|<\rho,
\label{eq:antitrap_drift}\\
C_1\le\E[\|\varepsilon_{n+1}\|^2\mid\mathcal F_n]\le C_2.
\label{eq:antitrap_variance}
\end{gather}
Then $\P(\Gamma\cap\{G_n\to0\})=0$.
\end{proposition}

\begin{proof}
For $k\ge2$, we define
\begin{equation}\label{eq:antitrap_truncation}
\widehat\varepsilon_k
:=\varepsilon_k\mathbf1_{\{\|\varepsilon_k\|\le\sqrt{k}\}}
-\E[\varepsilon_k\mathbf1_{\{\|\varepsilon_k\|\le\sqrt{k}\}}
\mid\mathcal F_{k-1}].
\end{equation}
We have $\E[\widehat\varepsilon_k\mid\mathcal F_{k-1}]=0$ and $\|\widehat\varepsilon_k\|\le2\sqrt{k}$. By the moment bound in~\eqref{eq:antitrap_noise}, we thus have
\begin{equation}\label{eq:antitrap_tail_summability}
\E\left[\sum_{k=2}^{\infty}\frac{1}{k}
\E[\|\varepsilon_k\|^2
\mathbf1_{\{\|\varepsilon_k\|>\sqrt{k}\}}\mid\mathcal F_{k-1}]
\right]
\le C\sum_{k=2}^{\infty}k^{-p/2}<\infty.
\end{equation}
We conclude that the series inside the expectation in~\eqref{eq:antitrap_tail_summability} is finite almost surely. Since
$$
\sum_{k=2}^{\infty}\P(\|\varepsilon_k\|>\sqrt{k})
\le\sum_{k=2}^{\infty}k^{-1}
\E[\|\varepsilon_k\|^2\mathbf1_{\{\|\varepsilon_k\|>\sqrt{k}\}}]
<\infty,
$$
by the Borel--Cantelli lemma, we have $\|\varepsilon_k\|\le\sqrt{k}$ for all sufficiently large $k$, almost surely. Using~\eqref{eq:antitrap_noise}, \eqref{eq:antitrap_truncation}, and~\eqref{eq:antitrap_tail_summability}, together with~\eqref{eq:antitrap_coefficients}, we obtain, for all sufficiently large $n$, almost surely,
\begin{align}
\sum_{k=n+1}^{\infty}c_k
\|\varepsilon_k-\widehat\varepsilon_k\|
&=\sum_{k=n+1}^{\infty}c_k
\bigl\|\E[\varepsilon_k\mathbf1_{\{\|\varepsilon_k\|>\sqrt{k}\}}
\mid\mathcal F_{k-1}]\bigr\|\notag\\
&\le\left(\sup_{k\ge n+1}kc_k^2\right)^{1/2}
\sum_{k=n+1}^{\infty}\frac{1}{k}
\E[\|\varepsilon_k\|^2\mathbf1_{\{\|\varepsilon_k\|>\sqrt{k}\}}
\mid\mathcal F_{k-1}]\notag\\
&=o\left(\left(\sum_{k=n+1}^{\infty}c_k^2\right)^{1/2}\right),
\label{eq:antitrap_centering_tail}\\
\sum_{k=n+1}^{\infty}c_k^2
\E[\|\varepsilon_k\|^2\mathbf1_{\{\|\varepsilon_k\|>\sqrt{k}\}}
\mid\mathcal F_{k-1}]
&\le\left(\sup_{k\ge n+1}kc_k^2\right)
\sum_{k=n+1}^{\infty}\frac{1}{k}
\E[\|\varepsilon_k\|^2\mathbf1_{\{\|\varepsilon_k\|>\sqrt{k}\}}
\mid\mathcal F_{k-1}]\notag\\
&=o\left(\sum_{k=n+1}^{2n}c_k^2\right).
\label{eq:antitrap_second_moment_tail}
\end{align}

For $t\ge0$, we write $n=\lfloor e^t\rfloor$ and define
\begin{equation}\label{eq:antitrap_embedding}
\begin{gathered}
\mathcal H_t:=\mathcal F_n,\quad
\widetilde F_t:=\frac{\upsilon_{n+1}F_n}{\log(n+1)-\log n},\quad
X_t:=G_n+(t-\log n)\widetilde F_t,\\
M_t:=\sum_{k=2}^{n}c_k\widehat\varepsilon_k,\quad
R_t:=\sum_{k=2}^{n}c_k
(r_k+\varepsilon_k-\widehat\varepsilon_k).
\end{gathered}
\end{equation}
After completing $(\mathcal H_t)_{t\ge0}$ in the usual way, we have a complete right-continuous filtration. By~\eqref{eq:antitrap_recursion}, we have
$$
X_t-X_0=\int_0^t\widetilde F_s\,\rmd s+M_t+R_t
\quad\text{for }t\ge0.
$$
Here $X$ is adapted and c\`adl\`ag, $\widetilde F$ is progressively measurable, $M$ is a c\`adl\`ag square-integrable martingale, and $R$ is an adapted c\`adl\`ag process with finitely many jumps on every bounded interval. We verify conditions \textup{(i)}--\textup{(v)} of Proposition~\ref{thm:RT_specialized} with $\alpha(t):=(\sum_{k=\lfloor e^t\rfloor+1}^{\infty}c_k^2)^{1/2}$. By~\eqref{eq:antitrap_coefficients}, this function is positive, nonincreasing and c\`adl\`ag, and tends to $0$ as $t\to\infty$.

For $t\notin\{\log k:k\ge1\}$, we have $X_{t-}=X_t$. Using~\eqref{eq:antitrap_drift} and~\eqref{eq:antitrap_embedding}, almost surely on $\Gamma$, for all sufficiently large $t$ such that $\|X_{t-}\|<\rho$, we have
\begin{align}
\|G_n\|&\le\|X_{t-}\|+\upsilon_{n+1}\|F_n\|<2\rho.\\
\langle X_{t-},\widetilde F_t\rangle &
=\frac{\upsilon_{n+1}\langle G_n,F_n\rangle}
{\log(n+1)-\log n}
+(t-\log n)\|\widetilde F_t\|^2\ge0.
\end{align}
We have verified condition \textup{(i)}.

Using~\eqref{eq:antitrap_noise}, we also have
\begin{equation}\label{eq:antitrap_variance_loss}
\begin{split}
0&\le\E[\|\varepsilon_k\|^2\mid\mathcal F_{k-1}]
-\E[\|\widehat\varepsilon_k\|^2\mid\mathcal F_{k-1}]\\
&=\E[\|\varepsilon_k\|^2\mathbf1_{\{\|\varepsilon_k\|>\sqrt{k}\}}
\mid\mathcal F_{k-1}]
+\bigl\|\E[\varepsilon_k\mathbf1_{\{\|\varepsilon_k\|>\sqrt{k}\}}
\mid\mathcal F_{k-1}]\bigr\|^2\\
&\le2\E[\|\varepsilon_k\|^2\mathbf1_{\{\|\varepsilon_k\|>\sqrt{k}\}}
\mid\mathcal F_{k-1}].
\end{split}
\end{equation}
Using~\eqref{eq:antitrap_variance}, \eqref{eq:antitrap_variance_loss}, and~\eqref{eq:antitrap_second_moment_tail}, together with $\lfloor e^{t+1}\rfloor\ge2n$, we have, almost surely on $\Gamma$, for all sufficiently large $t$,
\begin{equation}\label{eq:antitrap_brackets}
\begin{split}
\langle M\rangle_{t,t+1}
&=\sum_{k=n+1}^{\lfloor e^{t+1}\rfloor}c_k^2
\E[\|\widehat\varepsilon_k\|^2\mid\mathcal F_{k-1}]\\
&\ge C_1\sum_{k=n+1}^{\lfloor e^{t+1}\rfloor}c_k^2
-o\left(\sum_{k=n+1}^{2n}c_k^2\right)
\ge\frac{C_1}{2}\sum_{k=n+1}^{\lfloor e^{t+1}\rfloor}c_k^2,\\
\langle M\rangle_{t,\infty}
&\le C_2\sum_{k=n+1}^{\infty}c_k^2
=C_2\alpha(t)^2.
\end{split}
\end{equation}
Since $\alpha(t)^2-\alpha(t+1)^2=\sum_{k=n+1}^{\lfloor e^{t+1}\rfloor}c_k^2$, we have verified conditions \textup{(ii)} and \textup{(iii)} by~\eqref{eq:antitrap_brackets}.

For every $t\ge0$ and every $(\mathcal H_t)$-stopping time $S>t$, we have
$\|\Delta M_S\|\le2(\sup_{k\ge n+1}kc_k^2)^{1/2}$ on $\{S<\infty\}$, since every jump after $t$ has index greater than $n$. Consequently, we have
$$
\E[\|\Delta M_S\|^2\mathbf1_{\{S<\infty\}}\mid\mathcal H_t]
\le4\left(\sup_{k\ge n+1}kc_k^2\right)
\P(S<\infty\mid\mathcal H_t).
$$
Taking the finite deterministic c\`adl\`ag process $E_t:=4\sup_{k\ge\lfloor e^t\rfloor+1}kc_k^2=O(\alpha(t)^2)$ by~\eqref{eq:antitrap_coefficients}, we have verified condition \textup{(v)}.

By the Cauchy--Schwarz inequality, \eqref{eq:antitrap_centering_tail}, and the square summability of $(r_k)$ on $\Gamma$, we have, almost surely on $\Gamma$,
$$
\begin{aligned}
V(R,(t,\infty))
&\le\left(\sum_{k=n+1}^{\infty}c_k^2\right)^{1/2}
\left(\sum_{k=n+1}^{\infty}\|r_k\|^2\right)^{1/2}
+\sum_{k=n+1}^{\infty}c_k
\|\varepsilon_k-\widehat\varepsilon_k\|=o(\alpha(t)).
\end{aligned}
$$
We have verified condition \textup{(iv)}.

Almost surely on $\Gamma\cap\{G_n\to0\}$, using~\eqref{eq:antitrap_recursion}, \eqref{eq:antitrap_coefficients}, the eventual bound $\|\varepsilon_{n+1}\|\le\sqrt{n+1}$, and $\sum_{k=1}^{\infty}\|r_k\|^2<\infty$, we have
$$
\begin{aligned}
&\upsilon_{n+1}F_n
=G_{n+1}-G_n-c_{n+1}(\varepsilon_{n+1}+r_{n+1})
\rightarrow0,\\
&\sup_{\log n\le t<\log(n+1)}\|X_t-G_n\|
\le\upsilon_{n+1}\|F_n\|\rightarrow0,
\end{aligned}
$$
where we have also used~\eqref{eq:antitrap_embedding}. By Proposition~\ref{thm:RT_specialized}, we conclude that
\begin{equation*}
\P(\Gamma\cap\{G_n\to0\})
\le\P(\Gamma\cap\{X_t\to0\})=0.\qedhere
\end{equation*}
\end{proof}

 \section{Multivariate martingale central limit theorems}\label{appendix:martingale}
 \begin{proposition}\label{thm.mtg}
     Let $(X_{n,i}, \mathcal{F}_{n,i})_{1\le i\le k_n, n\ge 1}$ be a martingale difference array in $\R^d$ satisfying the following condition:
\begin{itemize}
    \item For every $\eps > 0$,
$$
\sum_{i=1}^{k_n} \E\left[ X_{n,i} X_{n,i}^\top \mathbf{1}_{\{\norm{X_{n,i}} > \eps\}} \mid \mathcal{F}_{n,i-1} \right] \convp 0.
$$
\item There exists a non-random positive semi-definite matrix $\Sigma \in \R^{d \times d}$ such that:
$$
\sum_{i=1}^{k_n} \E[X_{n,i}X_{n,i}^\top \mid \mathcal{F}_{n,i-1}] \convp \Sigma.
$$
\end{itemize} 
Then $$\sum_{i=1}^{k_n} X_{n,i} \convd \mathcal{N}_d(0, \Sigma).$$
 \end{proposition}
\begin{proof}
    For any fixed $v \in \R^d$, define $Y_{n,i} = v^\top X_{n,i}$.
    We have
     $$
\sum_{i=1}^{k_n} \E[Y_{n,i}^2 \mid \mathcal{F}_{n,i-1}]=v^\top\sum_{i=1}^{k_n} \E[X_{n,i}X_{n,i}^\top \mid \mathcal{F}_{n,i-1}]v \convp v^\top\Sigma v.
$$  
By Cauchy-Schwarz inequality $|Y_{n,i}|=|v^\top X_{n,i}|\le \|v\|\cdot \|X_{n,i}\|$. For each $\eps>0$, we thus have
$$
\sum_{i=1}^{k_n} \E\left[ Y_{n,i}^2 \mathbf{1}_{\{|{Y_{n,i}}| > \eps\}} \mid \mathcal{F}_{n,i-1} \right]\le 
v^\top\sum_{i=1}^{k_n} \E\left[ X_{n,i} X_{n,i}^\top  \mathbf{1}_{\{\|{X_{n,i}}\|\cdot \|v\| > \eps\}} \mid \mathcal{F}_{n,i-1} \right] v \convp 0.
$$
Applying the uni-variate martingale CLT to $(Y_{n,i})$, we infer that $\sum_{i=1}^{k_n} Y_{n,i} \convd \mathcal{N}(0, v^\top \Sigma v)$ (see e.g. Corollary 3.1 in \cite{HH1980}). By the Cramér-Wold theorem, we obtain the desired result. 
\end{proof}
 \begin{corollary} \label{thm.mtg2} Let $(M_n)$ be a square-integrable $\R^d$-valued martingale with $M_0=0$, and write $\Delta_k:=M_k-M_{k-1}$. Suppose that
        $$\Lambda_n := \sum_{k=1}^n \E[\Delta_k \Delta_k^\top] = \E[M_n M_n^\top],\quad \lambda_{\min}(\Lambda_n)\to\infty$$ and $\Lambda_n$ is eventually positive definite.
Assume that 
for every $\eps > 0$,
\begin{align}
    &\Lambda_n^{-1/2}\sum_{k=1}^n \E\Big[ \Delta_k \Delta_k^\top \,
        \mathbf{1}_{\Big\{\norm{\Lambda_n^{-1/2}\Delta_k} > \eps\Big\}} \,\Big|\, \mathcal{F}_{k-1} \Big]\Lambda_n^{-1/2}
        \convp 0\quad \text{and}\\
    & \Lambda_n^{-1/2} \sum_{k=1}^n \E\left[\Delta_k \Delta_k^\top \mid \mathcal{F}_{k-1}\right] \Lambda_n^{-1/2} \convp I_d.
\end{align}
Then,
$\Lambda_n^{-1/2} M_n \convd \mathcal{N}_d(0, I_d)$.
 \end{corollary}

 \begin{proof} We apply Proposition \ref{thm.mtg} to
     $X_{n,k} = \Lambda_n^{-1/2} \Delta_k$ and $\mathcal{F}_{n,k}=\mathcal{F}_k$ for $k = 1, \dots, n.
$
 \end{proof} 

 \section{Robbins-Siegmund theorem} \label{appendix:RSthm}
 \begin{proposition}[Robbins-Siegmund 1971, \cite{RS1971}]\label{RS.thm}
     Let $(X_n)_{n\ge0}, (Y_n)_{n\ge0},  (\alpha_n)_{n\ge0}, (\beta_n)_{n\ge0}$ be sequences of nonnegative random variables adapted to $(\mathcal{F}_n)$ such that
     $$\mathbb{E}[X_{n+1} \mid \mathcal{F}_n] \le (1 + \alpha_n) X_n + \beta_n - Y_n, \quad \text{a.s.}$$
Then,  on the event $\Big\{\sum_{n=0}^{\infty} \alpha_n < \infty \quad \text{and} \quad \sum_{n=0}^{\infty} \beta_n < \infty\Big\}$, we have:
 \begin{itemize}
     \item  $X_{n}$ converges a.s. to a finite random variable $X_{\infty}$;
     \item  $\sum_{n=0}^{\infty} Y_n < \infty$ a.s.
 \end{itemize}
 \end{proposition}

\section{Several results for stochastic differential equations}\label{appendix:sde}

Let $(\Omega,\mathcal{F},\P)$ be a complete probability space with a complete right-continuous filtration $(\mathcal{F}_t)_{t\ge 0}$, and let 
	$(B_t)_{t\ge 0}$ be a standard Brownian motion on $\mathbb{R}^{d}$. Throughout this section, we consider a time-homogeneous SDE in $\mathbb{R}^d$:
	\begin{equation}\label{eq:sde_appendix}
		\mathrm{d}X_t = b(X_t)\,\mathrm{d}t 
		+ \sigma(X_t)\,\mathrm{d}B_t,
	\end{equation}
	with generator
	$$
	(\mathcal{L}\varphi)(x)
	= \langle b(x),\,\nabla\varphi(x)\rangle
	+\frac{1}{2}\mathrm{tr}\!\left(
	\sigma(x)\sigma(x)^\top\nabla^2\varphi(x)\right) \quad \text{for }\varphi\in C^2_c(\R^d).$$
    
\subsection{Non-explosion}

	The following results are taken from 
	\cite{MR1234295}.
		A function $V:\mathbb{R}^d\to\mathbb{R}_+$ is called 
		\emph{norm-like} if $V(x)\to+\infty$ as $\|x\|\to\infty$, 
		i.e., the level sets $\{x: V(x)\le h\}$ are precompact 
		for every $h>0$.
	\medskip
	\begin{theorem}[\cite{MR1234295}, Theorem~2.1: 
		Non-explosion criterion]\label{thm.Ber_nonexp}
		Suppose that the SDE has a local solution up to its explosion time. Assume that there exist a nonnegative norm-like function $V\in C^2(\R^d)$ and 
		constants $c,d>0$ such that
		$$
		(\mathcal{L}V)(x) \le cV(x)+d
		\quad\text{for }x\in\mathbb{R}^d.
		$$
		Then the explosion time is almost surely infinite for every starting point $x\in\mathbb R^d$.
	\end{theorem}
	
\subsection{Gaussian bounds for the transition density}\label{sec:gaussian_bounds}
Consider the
$d$-dimensional SDE
\begin{equation}\label{eq:sde_MPZ}
	\mathrm{d}X_t = b(X_t)\,\mathrm{d}t + \sigma(X_t)\,\mathrm{d}B_t\quad\text{for } t\ge0,\quad X_0=x.
\end{equation}




For $\varepsilon\in(0,1]$, let $\rho$ be a nonnegative smooth function
supported in the unit ball of $\mathbb{R}^d$ with $\int_{\mathbb{R}^d}\rho(x)\,\mathrm{d}x=1$,
and define the spatial mollification of $b$ at scale $\varepsilon$ by
\begin{equation}\label{eq:mollification}
	b^{(\varepsilon)}(x)
	:= \bigl(b\ast\rho_\varepsilon\bigr)(x)
	= \int_{\mathbb{R}^d} b(y)\,\rho_\varepsilon(x-y)\,\mathrm{d}y \quad \text{with}\quad \rho_\varepsilon(x) := \varepsilon^{-d}\rho(\varepsilon^{-1}x)
\end{equation}
Under condition $(H_b^0)$ below, $b$ has at most linear growth, and its mollification is smooth with at most linear growth. Therefore the
ordinary differential equation
\begin{equation}\label{eq:flow}
	\dot\theta^{(\varepsilon)}_{t}(x)
= b^{(\varepsilon)}\bigl(\theta^{(\varepsilon)}_t(x)\bigr),
	\quad
	\theta^{(\varepsilon)}_{0}(x) = x,
\end{equation}
admits a unique global solution, and the map
$x\mapsto\theta^{(\varepsilon)}_{t}(x)$ is a $C^\infty$-diffeomorphism
for each $t$.


We will use the following result in the
proof of Lemma~\ref{sde.unique_solution}

\begin{theorem}[\cite{MPZ2021}, Theorem~1.2(i)]\label{thm:MPZ}
    Assume that the following two conditions hold.
    \begin{itemize}
        \item Condition $(H^{\alpha}_\sigma)$:
        There exist constants $\kappa_0\ge1$ and $\alpha\in(0,1)$ such that
        \begin{align*}
       & \kappa_0^{-1}|\xi|^2
        \le \langle\sigma(x)\sigma(x)^\top\xi,\xi\rangle
        \le \kappa_0|\xi|^2\quad\text{and}\quad |\sigma(x)-\sigma(y)|\le\kappa_0|x-y|^{\alpha}
        \quad\text{for all }x,y,\xi \in\mathbb{R}^d.
        \end{align*}

        \item Condition $(H^{0}_b)$:
        There exists a constant $\kappa_1>0$ such that
        $
        |b(0)|\le\kappa_1,
        $
        and
        $$
        |b(x)-b(y)|
        \le
        \kappa_1\bigl(1\vee |x-y|\bigr)
        \quad\text{for all }x,y\in\mathbb{R}^d.
        $$
    \end{itemize}
    Then the SDE \eqref{eq:sde_MPZ} admits a unique weak solution. Moreover, for any
    $T>0$, there exist constants $C_0\ge1$ and
    $\lambda_0\in(0,1]$, depending only on
    $T,\alpha,\kappa_0,\kappa_1$ and $d$, such that for all
    $0<t\le T$ and $x,y\in\mathbb{R}^d$, the transition density
    $p_t(x,y)$ exists and satisfies the two-sided Gaussian bound
    \begin{equation}\label{eq:two_sided_MPZ}
        C_0^{-1}\,
        g_{\lambda_0^{-1}}\!\bigl(t,\theta^{(1)}_{t}(x)-y\bigr)
        \le
        p_t(x,y)
        \le
        C_0\,
        g_{\lambda_0}\!\bigl(t,\theta^{(1)}_{t}(x)-y\bigr),
    \end{equation}
    where
    $
    g_\lambda(t,z):=t^{-d/2}\exp(-\lambda|z|^2/t),
    $
    and $\theta^{(1)}_t(x)$ is the mollified flow defined in
    \eqref{eq:flow} with $\varepsilon=1$.
\end{theorem}

\subsection{Exponential ergodicity} \label{sec:Ber2021}
Let $(X_t)_{t\ge0}$ be a time-homogeneous Markov process on
$\mathbb{R}^d$ with transition kernel $\mathscr{P}(x, \cdot)$.
For a measurable function $f:\mathbb{R}^d\to\mathbb{R}$ and a function $V:\mathbb{R}^d\to[0,\infty)$, define the weighted norm
\begin{equation}\label{eq:V_norm}
    \|f\|_{1+V} := \sup_{x\in\mathbb{R}^d}
\frac{|f(x)|}{1+V(x)}.
\end{equation}
\begin{theorem}[\cite{MR2857021}, Theorem~1.2:
	Exponential ergodicity in discrete time]\label{thm:Ber2021}
	Assume that the following conditions hold:
    \begin{itemize}
        \item[a.] 
        Geometric drift condition:     
 There exist a function $V:\mathbb{R}^d\to[0,\infty)$ and
	constants $d\ge0$ and $\gamma\in(0,1)$ such that
	\begin{equation}\label{eq:FL_Ber}
		(\mathscr{P}V)(x) \le \gamma V(x) + d
		\quad\text{for all }x\in\mathbb{R}^d.
	\end{equation}
    \item[b.] 
    Minorisation condition: Let $C=\{x\in\mathbb{R}^d: V(x)\le R\}$ for some
	$R>2d(1-\gamma)^{-1}$.
	Then there exist $\alpha\in(0,1)$ and a probability measure
	$\nu$ such that
	\begin{equation}\label{eq:minor_Ber}
		\inf_{x\in C}\mathscr{P}(x,A) \ge \alpha\,\nu(A)
	\end{equation}
	holds for all Borel sets $A\subset\mathbb{R}^d$.
    \end{itemize}    
Then $\mathscr{P}$ admits a unique invariant probability measure
	$\pi$. Furthermore, there exist constants $M>0$ and
	$\gamma\in(0,1)$ such that
	\begin{equation}\label{eq:exp_erg}
		\|\mathscr{P}^n f - \pi(f)\|_{1+V}
		\le M\gamma^n\,\|f-\pi(f)\|_{1+V}
	\end{equation}
	holds for all measurable functions $f:\mathbb{R}^d\to\mathbb{R}$
	such that $\|f\|_{1+V}<\infty$.
\end{theorem}

\begin{proposition}\label{thm.ergodicity}
       Assume that 	                    the two
     conditions $(H_\sigma^{\alpha})$ and $(H_b^0)$ of Theorem \ref{thm:MPZ} hold and there exists a nonnegative norm-like function $V\in C^2(\R^d)$ such that
    $$\mathcal{L}V(x)\le -c_0 V(x)+d_0\quad \quad\text{for all } x\in \R^d,$$
with some positive constants $c_0$ and $d_0$. Then the SDE admits a unique invariant probability measure $\pi$. Furthermore, there exist constants $C, \lambda>0$ such that
	\begin{equation*}
		|{P}_t f(x) - \pi(f)|
		\le C {\rm e}^{-\lambda t}(1+V(x))\|f\|_{\infty}.
	\end{equation*}
 for each bounded measurable function $f$, for each $x\in \R^d$ and $t\ge0$.
\end{proposition}
\begin{proof}
	We verify the conditions (a) and (b) of Theorem~\ref{thm:Ber2021} for the skeleton chain $(X_{n\delta})_{n\ge0}$ with some fixed $\delta>0$.	
  
\noindent\textbf{a. Geometric drift condition:}
We fix $x\in\mathbb R^d$ and, for $R>\|x\|$, define
$\tau_R:=\inf\{s\ge0:\|X_s\|\ge R\}$. Using the Dynkin formula and the assumption that $\mathcal{L}V(x)\le -c_0 V(x)+d_0$ for all $x\in \R^d$, we obtain, for every $t\ge0$,
\begin{equation}\label{eq:ergodic_stopped_Dynkin}
\begin{split}
\E^x\!\left[e^{c_0(t\wedge\tau_R)}
V(X_{t\wedge\tau_R})\right]
&=V(x)+\E^x\!\left[
\int_0^{t\wedge\tau_R}e^{c_0s}
\bigl(\mathcal LV(X_s)+c_0V(X_s)\bigr)\rmd s
\right]\\
&\le V(x)+d_0\E^x\!\left[
\int_0^{t\wedge\tau_R}e^{c_0s}\rmd s\right]\le V(x)+\frac{d_0}{c_0}(e^{c_0t}-1).
\end{split}
\end{equation}
By Theorem~\ref{thm.Ber_nonexp}, we have $\tau_R\uparrow\infty$ almost surely as $R\to\infty$. Using $V\ge0$ and Fatou's lemma in~\eqref{eq:ergodic_stopped_Dynkin}, we obtain
\begin{equation}\label{eq:ergodic_semigroup_drift}
P_tV(x)\le e^{-c_0t}V(x)
+\frac{d_0}{c_0}(1-e^{-c_0t})
\quad\text{for all } x\in\mathbb R^d,\quad t\ge0.
\end{equation}
Taking $t=\delta$ verifies the condition \textup{(a)} of Theorem~\ref{thm:Ber2021} for $\mathscr P=P_\delta$, with
$$
\gamma:=e^{-c_0\delta}\in(0,1),
\quad d:=\frac{d_0}{c_0}(1-e^{-c_0\delta})<\infty.
$$

	\noindent\textbf{b. Minorisation condition:}
	Fix $R>2d(1-\gamma)^{-1}$ and let
	$C=\{x\in\mathbb{R}^d:V(x)\le R\}$ as in
	the condition (b). Note moreover that
	\[
	\frac{2d}{1-\gamma}
	= \frac{2\cdot\frac{d_0}{c_0}(1-e^{-c_0\delta})}
	{1-e^{-c_0\delta}}
	= \frac{2d_0}{c_0},
	\]
	so the compact set $C=\{x\in\mathbb{R}^d:V(x)\le R\}$ required in
	the condition (b) is well-defined for any fixed
	$R>\frac{2d_0}{c_0}$, independently of $\delta$.
        
	Since the two conditions $(H^\alpha_\sigma)$ and $(H^0_b)$ hold, we apply Theorem~\ref{thm:MPZ} to obtain a lower bound
	for the transition density $p_\delta(x,y)$. 
By Theorem~\ref{thm:MPZ},
	the transition density $p_\delta(x,y)$ satisfies
	\[
	p_\delta(x,y)
	\ge C_0^{-1}\,g_{\lambda_0^{-1}}\!\bigl(\delta,\,
	\theta^{(1)}_\delta(x)-y\bigr)
	\quad\text{for all }x,y\in\mathbb{R}^d,
	\]
	where $g_\lambda(t,z)=t^{-d/2}\exp(-\lambda|z|^2/t)$ and
	$\theta^{(1)}_\delta(x)$ is the
	mollified flow defined in \eqref{eq:flow}.
	Let $B_\rho\subset\mathbb{R}^d$ be any open ball of radius $\rho>0$.
	Since $C$ is compact and $x\mapsto \theta_\delta^{(1)}(x)$ is continuous,
	the set $\theta_\delta^{(1)}(C)$ is compact. Therefore the function
	$
	(x,y)\mapsto
C_0^{-1}\,g_{\lambda_0^{-1}}\!\bigl(\delta,\theta_\delta^{(1)}(x)-y\bigr)
	$
	is continuous and strictly positive on the compact set
	$C\times \overline{B_\rho}$. Hence
	$
	\kappa:=\inf_{x\in C,\;y\in B_\rho}
C_0^{-1}\,g_{\lambda_0^{-1}}\!\bigl(\delta,\theta_\delta^{(1)}(x)-y\bigr)
	>0$.
	Choosing $\rho>0$ sufficiently small, we assume that
	$
	\varepsilon_0:=\kappa\,|B_\rho|\in(0,1)$,
	where $|\cdot|$ denotes Lebesgue measure on $\mathbb{R}^d$. Define the probability measure
	$
	\nu(A):=|A\cap B_\rho|/|B_\rho|$,
	for $A\in\mathcal{B}(\mathbb{R}^d).$
	Then, for every $x\in C$ and every Borel set $A$,
	\begin{align*}
		P_\delta(x,A)
		&=
		\int_A p_\delta(x,y)\,\rmd y
		\ge
		\int_{A\cap B_\rho} p_\delta(x,y)\,\rmd y \ge
		\kappa\,|A\cap B_\rho|
		=
		\varepsilon_0\,\nu(A).
	\end{align*}
Hence the condition (b) of Theorem~\ref{thm:Ber2021} holds.
    \medskip
	
We have thus verified the conditions \textup{(a)} and \textup{(b)} of Theorem~\ref{thm:Ber2021}. Hence the skeleton chain $(X_{n\delta})_{n\ge0}$ admits a unique invariant probability measure $\pi$, and there exist constants $M>0$ and $\bar\gamma\in(0,1)$ such that
\begin{align}\label{inq.ergo.V}
\|P_{n\delta}f-\pi(f)\|_{1+V}
\le M\bar\gamma^n\,\|f-\pi(f)\|_{1+V}
\end{align}
for all $n\ge0$ and all measurable $f$ with $\|f\|_{1+V}<\infty$. Here, $\|\cdot\|_{1+V}$ is defined by~\eqref{eq:V_norm}.

For each $r\ge0$, the measure $\pi P_r$ is invariant under $P_\delta$, since the semigroup operators commute:
$$
(\pi P_r)P_\delta=(\pi P_\delta)P_r=\pi P_r.
$$
By uniqueness of the invariant probability measure for the skeleton chain, we have $\pi P_r=\pi$. Hence $\pi$ is invariant for the full semigroup. Any invariant probability measure for the full semigroup is also invariant under $P_\delta$, so it equals $\pi$.

For bounded measurable $f$, using the fact that $\|f-\pi(f)\|_{1+V}=\sup_{x\in\mathbb{R}^d}
	\frac{|f(x)-\pi(f)|}{1+V(x)}\le2\|f\|_\infty$ and~\eqref{inq.ergo.V}, we have
\begin{align}\label{inq.ergo.V2}
|P_{n\delta}f(x)-\pi(f)|
\le2M\bar\gamma^n(1+V(x))\|f\|_\infty.
\end{align}
Finally, for each $t\ge0$, we write $t=n\delta+r$ with $n=\lfloor t/\delta\rfloor$ and $r\in[0,\delta)$. Using the semigroup property, the invariance of $\pi$ and~\eqref{inq.ergo.V2}, we obtain
\begin{align*}
|P_t f(x)-\pi(f)|
&=|P_{n\delta}(P_r f)(x)-\pi(P_r f)|\\
&\le2M\bar\gamma^n(1+V(x))\|P_r f\|_\infty\le2M\bar\gamma^n(1+V(x))\|f\|_\infty,
\end{align*}
where in the last step we have used the fact that $\|P_r f\|_\infty\le\|f\|_\infty$.
Since $n\ge t/\delta-1$, we have
$\bar\gamma^n\le\bar\gamma^{-1}e^{-\lambda t}$ with $\lambda:=-\frac{\log\bar\gamma}{\delta}>0.$
Taking $C:=2M\bar\gamma^{-1}$, we conclude that
$$
|P_t f(x)-\pi(f)|
\le Ce^{-\lambda t}(1+V(x))\|f\|_\infty.
$$
This completes the proof.
\end{proof}

\end{appendix}

\section*{Acknowledgements}
The authors would like to thank Gunter M. Sch\"utz for his kind comments and for pointing out relevant literature on the elephant random walk. Tuan-Minh Nguyen was partially supported by the Australian Research Council under grant ARC DP230102209.

\bibliographystyle{amsplain}
\bibliography{refs}

\end{document}